\documentclass[a4paper]{amsart}
\usepackage{etex}
\usepackage{enumerate}
\usepackage[english]{babel}

\usepackage{graphicx}
\usepackage{multicol}
\usepackage[bottom]{footmisc}
\usepackage[export]{adjustbox}
\usepackage{functan}
\usepackage[latin1]{inputenc}
\usepackage{cancel}
\usepackage{calligra}
\usepackage{colortbl}
\usepackage{multirow}
\usepackage{enumerate}
\usepackage{varioref}
\usepackage{booktabs}
\usepackage{amscd}
\usepackage{color}
\usepackage{colortbl}
\usepackage{pstricks}
\usepackage{pst-all}
\usepackage{mparhack}
\usepackage{amssymb}
\usepackage{amsmath}
\usepackage{dsfont}
\usepackage{mathrsfs}
\usepackage{amsfonts}
\usepackage{amssymb}
\usepackage[mathcal]{euscript}
\usepackage[all,cmtip]{xy}
\usepackage{amssymb}
\usepackage{amsthm}
\theoremstyle{definition}
\newtheorem{definition}{Definition}[section]
\newtheorem{example}[definition]{Example}
\newtheorem{remark}[definition]{Remark}

\theoremstyle{plain}

\newtheorem{theorem}[definition]{Theorem}
\newtheorem{proposition}[definition]{Proposition}
\newtheorem{lemma}[definition]{Lemma}
\newtheorem{corollary}[definition]{Corollary}

\numberwithin{equation}{section}

\def \alt96 {`}
\def \RN {\mathds{R}^N}
\def \R {\mathds{R}}

\def \loc {\mathrm{loc}}

\def \meas {\mathrm{meas}}
\def \N {\mathds{N}}

\def \G {\mathbb{G}}
\def \LieG {\mathrm{Lie}(\G)}

\def \dela {{\delta_\lambda}}

\def \LL {{\mathcal{L}}}
\def \Ht {\mathcal{H}}
\def \d {\mathrm{d}}
\def \de {\partial}

\def \meas {\mathrm{meas}}

\def \longto {\longrightarrow}
\def \txt {\textstyle}

\def \LL {\mathcal{L}}

\begin{document}
  \author{Stefano Biagi}
 \address{Stefano Biagi: Dipartimento di Ingegneria Industriale e Scienze Matematiche,
 Università Po\-li\-tec\-ni\-ca delle Marche,  via Brecce Bianche, 12, I-60131 Ancona, Italy.}
 \email{s.biagi@dipmat.univpm.it}
  \author{Andrea Bonfiglioli}
 \address{Andrea Bonfiglioli: Dipartimento di Matematica, Università di Bologna,
 Piazza Porta San Donato, 5, I-40126 Bologna, Italy.}
 \email{andrea.bonfiglioli6@unibo.it}
 \title[Heat kernels for homogeneous H\"ormander PDOs]{Global Heat Kernels for \\ Parabolic Homogeneous H\"ormander Operators} 
 
 \begin{abstract}
 The aim of this paper is to prove the existence and several selected properties
 of a global fundamental Heat kernel $\Gamma$ for the parabolic operators $\Ht=\sum_{j=1}^m
 X_j^2-\de_t$, where $X_1,\ldots,X_m$
  are smooth vector fields on $\R^n$ satisfying H\"ormander's
  rank condition, and enjoying a suitable homogeneity assumption with respect to a family of non-isotropic
  dilations.
  The proof of the existence of $\Gamma$ is based on a (algebraic) global lifting technique, together
  with a re\-pre\-sen\-ta\-tion of  $\Gamma$ in terms of the integral (performed over the lifting variables)
  of the Heat kernel for the Heat operator associated with a suitable sub-Laplacian
  on a homogeneous Carnot group.
  Among the features of $\Gamma$ we prove:
  homogeneity and symmetry properties; summability properties;
  its va\-ni\-shing at infinity; the uniqueness of the bounded solutions of the related Cauchy problem;
  reproduction and density properties; an integral re\-pre\-sen\-ta\-tion
  for the higher-order derivatives. \medskip
  
   \medskip
  
\end{abstract}
\maketitle
\section{Introduction}\label{sec:intro}
 Given a certain class of H\"ormander PDOs (Partial Differential Operators, here and throughout),
 the availability of some integral representation formulas
 for an associated global fundamental solution $\Gamma$ and for its derivatives in terms of
 well-behaved kernels defined on richer higher dimensional structures (such as homogeneous Carnot groups)
 can lead to global pointwise estimates of $\Gamma$ and of its derivatives, only via very general results on the geometry of H\"ormander operators;
 see e.g., the recent investigation \cite{BiBoBra}.
 A considerable amount
 of work needs to be accomplished in order to obtain both the existence of a global $\Gamma$
 and of well-behaved representation formulas, as shown in \cite{BiagiBonfiglioliLast}.

 The aim of the present study is to accomplish this work for a class of Heat-type evolution PDOs not contained
 in the stationary case faced in \cite{BiagiBonfiglioliLast}.
 As the approach in the latter paper proved fruitful, we shall try to adapt
 some ideas therein contained to the evolutive case; this programme is
 complicated by the preliminary need for a Gaussian behavior of the lifted Heat kernels
  (see e.g., \cite{JerisonSanchezCalle,KusuokaStroock1,KusuokaStroock2,VaropoulosSaloffCosteCoulhon}).\medskip

 On the other hand, the parabolic setting features interesting problems, such as the study of the initial Cauchy problem,
 and the richer properties of the associated potentials. On the horizon of the present work, we expect to
 investigate Gaussian pointwise estimates of the Heat kernel here constructed.\medskip

 To be more explicit, the aim of this paper is to prove the existence of a well-behaved
  global fundamental solution $\Gamma$ (also referred to as a Heat kernel) for the
  (degenerate) evolution Heat-ty\-pe PDOs  $\Ht$ of the form
\begin{equation}\label{Hiniziale}
  \Ht = \sum_{j = 1}^mX_j^2
  - \frac{\de}{\de t} \qquad \text{on $\R^{1+n} = \R_t\times\R^n_x$},
\end{equation}
  where $X_1,\ldots,X_m$
  are smooth vector fields on $\R^n_x$ satisfying H\"ormander's
  rank condition in space $\R^n_x$, and enjoying a suitable homogeneity assumption w.r.t.\,a family of non-isotropic
  dilations, which we shall describe subsequently.
  Our approach is two-fold: it relies on a (algebraic) global `\emph{lifting}' procedure, and on an integral `\emph{saturation}' technique.
 Roughly put, we construct a lifting operator $\widetilde{\Ht}$ for $\Ht$ of the form
\begin{equation}\label{Hiniziale2}
 \widetilde{\Ht} = \sum_{j = 1}^m \Big(X_j(x)+R_j(x,\xi)\Big)^2
  - \frac{\de}{\de t} \qquad \text{on $\R^{1+n+p} = \R_t\times\R^n_x\times \R^p_\xi$},
\end{equation}
 where $R_1(x,\xi),\ldots,R_m(x,\xi)$ are vector fields operating
 only in the variables $\xi=(\xi_1,\ldots,\xi_p)$ (with coefficients possibly depending on
 $(x,\xi)\in \R^n\times\R^p$), in such a way that
 the existence of a  \emph{global} (i.e., defined throughout $\R^{1+n+p}$) fundamental solution $\widetilde{\Gamma}$
 for $\widetilde{\Ht}$
 be ensured. Then, we want to redeem a fun\-da\-men\-tal solution
 $\Gamma$ for $\Ht$ by integrating $\widetilde{\Gamma}$ over the lifting variables
 $\xi\in \R^p$; to this end, it is necessary to know that
 $\widetilde{\Gamma}$ be globally integrable w.r.t.\,$\xi\in \R^p$, which is one of the crucial points
 of our approach.
 We refer to this integration procedure as a `saturation' argument.

 In the analysis of fundamental solutions for linear PDOs, the idea of passing through a lifting procedure
 and a saturation of the lifting variables is certainly not new, and it traces back to
 Rothschild and Stein's pivotal paper \cite{RothschildStein} (see also Nagel, Stein, Wainger \cite{NSW}); however,
 Rothschild and Stein's lifting is a local tool, whereas, as we stressed, we need a global
 technique since we aim to obtain fundamental solutions defined on the whole space (and vanishing at infinity).
 Global integrability (at infinity) over the saturation variables is a non-trivial fact.
 We shall describe in a moment how we face these problems. Incidentally, we observe that
 in \cite{RothschildStein} only suitable parametrices of a fundamental solution are studied, which again reflects the local/approximation
 nature of the lifting in \cite{RothschildStein}.\medskip

 The basic idea of obtaining
 fundamental solutions for Heat-type o\-pe\-ra\-tors via saturation
 ar\-gu\-ments is very well described in the Euclidean setting.
  Indeed, it is well known that
  a global fundamental solution (with pole at the origin of $\R^{1+n}$) for the classical
  Heat operator $\Ht_n := \Delta_n - \de/\de t$ on $\R^{1+n}$ is given by
  (we use the notation $\chi_A$ for the indicator
  function of a set $A$):
  $$\Gamma_n(t,x) = \chi_{(0,\infty)}(t)\,
  \frac{1}{(4\,\pi\,t)^{n/2}}\,\exp\left(-\frac{\sum_{j=1}^n x_j^2}{4\,t}\right), \qquad
  (t,x)\in\R\times\R^n.$$
  Then, if we consider the Heat operator $\mathcal{H}_{n+p}$  on $\R^{1+n+p}$ and if we integrate
  its fundamental solution $\Gamma_{n+p}$ (with pole at the origin of $\R^{1+n+p}$) with respect to the last
  $p$ variables, we obtain (upon the trivial fact $\int_\R \exp(-\frac{\xi^2}{4t})\,\d\xi=\sqrt{4\pi t}\,$)
  \begin{align*}
   & \int_{\R^p}\Gamma_{n+p}(t,x,\xi)\,\d\xi \\
   & \quad = \chi_{(0,\infty)}(t)\,
  \frac{1}{(4\pi\,t)^{(n+p)/2}}\,
  \exp\left(-\frac{\sum_{j=1}^n x_j^2}{4\,t}\right)\int_{\R^p}
  \exp\left(-\frac{\sum_{j=1}^p \xi_j^2}{4\,t}\right)\,\d \xi \\
  & \quad = \chi_{(0,\infty)}(t)\,
  \frac{1}{(4\,\pi\,t)^{n/2}}\,
  \exp\left(-\frac{\sum_{j=1}^n x_j^2}{4\,t}\right) = \Gamma_n(t,x).
  \end{align*}
  In other words, the Heat kernel $\Gamma_{n}$ of
  $\Ht_{n}$ can be recovered by the Heat kernel $\Gamma_{n+p}$ of $\Ht_{n+p}$
   by a saturation technique:
  $$\Gamma_n(t,x)=  \int_{\R^p}\Gamma_{n+p}(t,x,\xi)\,\d\xi, \qquad
  (t,x)\in\R\times\R^n. $$
  A global lifting/saturation process may likely occur in other interesting cases (for non-elliptic
   o\-pe\-ra\-tors):
  see e.g., Bauer, Furutani, Iwasaki \cite{BauerFurutaniIwasaki};
  Calin, Chang, Furutani, Iwasaki \cite[Sect. 10.3]{CalinChangFurutaniIwasaki};
  Beals, Gaveau, Greiner, Kannai \cite{BealsGaveauGreinerKannai2}.
  Explicit formulas for some Heat kernels on nilpotent Lie
  groups can be found in:
  Agrachev, Boscain, Gauthier, Rossi
  \cite{AgrachevBoscainGauthierRossi};
   Beals, Gaveau, Greiner
  \cite{BealsGaveauGreiner1,BealsGaveauGreiner2};
  Boscain, Gauthier, Rossi
  \cite{BoscainGauthierRossi};
   Cygan \cite{Cygan};
   Furutani \cite{Furutani};
   Gaveau \cite{Gaveau}.\medskip

  The same process was exploited in the paper \cite{BiagiBonfiglioliLast}, which
  provides some general structural assumptions
  showing when lifting/saturation can be suc\-ces\-sful\-ly applied (see Theorem \ref{exTheoremA}).
  We fix once and for all the definition
  of a lifting of a PDO $P$, while postponing
  the precise notion
  of a global fundamental solution $\Gamma$ to Theorem \ref{mainteo};
  for the time being, by $\Gamma$ we mean a function of two variables $(z;\zeta)\in \R^{1+n}\times\R^{1+n}$
  (the first of which is called the `pole')
  such that, for any fixed pole $z$, we have $P(\Gamma(z;\cdot))=-\mathrm{Dir}_z$ in the
  weak sense of distributions ($\mathrm{Dir}_z$ is the Dirac mass at $z$).

  In order to distinguish it from the local Rothschild and Stein's lifting technique,
  we define a simpler notion of the lifting of $P$ as follows: if $P$ is a smooth linear PDO on $\R^{1+n}_z$, we say that
  the PDO $\widetilde{P}$ defined on $\R^{1+n}_z\times\R^p_\xi$  is a lifting of $P$
  (or simply that $\widetilde{P}$ lifts $P$) if:
  \begin{itemize}
    \item $\widetilde{P}$ has smooth coefficients, possibly
    depending on $(z,\xi)\in\R^{1+n}\times\R^p$;
    \item for every fixed $f \in C^\infty(\R_z^{1+n})$, one has
  \begin{equation}\label{sec.one:def_eqLliftok}
    \widetilde{P}(f\circ \pi)(z,\xi) = (P f)(z),
    \quad \text{for every
    $(z,\xi) \in \R^{1+n}\times\R^p$},
  \end{equation}
   where $\pi(z,\xi)=z$ is the canonical
   projection of $\R^{1+n}\times\R^p$ onto $\R^{1+n}$.
   \end{itemize}
 For example, with this definition, $\widetilde{\Ht}$
 in \eqref{Hiniziale2} is a lifting of $\Ht$ in \eqref{Hiniziale}.

 In general, the idea of obtaining a fundamental solution $\Gamma$ for $P$ via a fundamental
 solution $\widetilde{\Gamma}$ for $\widetilde{P}$ by integration over the lifting $\R_\xi^p$-variables
 is natural but subtle, as we now describe. Let us start by writing down the definition of the distributional identity
\begin{equation}\label{distriheur}
 \widetilde{P}\Big\{(\zeta,\eta)\mapsto\widetilde{\Gamma}\Big((z,\xi);(\zeta,\eta)\Big)\Big\}=-\mathrm{Dir}_{(z,\xi)},
\end{equation}
 by first conveniently freezing the variable $\xi$ at $0\in \R^p$: this boils down to the identity (valid for every $\psi\in C_0^\infty(\R^{1+n+p})$ and every $(z,0)\in \R^{1+n+p}$)
\begin{equation}\label{distriheureee.EQ1}
 \int_{\R^{1+n}}\d \zeta \int_{\R^p}\d \eta\,\, \widetilde{\Gamma}\Big((z,0);(\zeta,\eta)\Big)\,
 \widetilde{P}^*\big(\psi(\zeta,\eta)\big)
  = -\psi(z,0).
\end{equation}
 Then, we aim to recover a fundamental solution $\Gamma$ for $P$ starting from identity \eqref{distriheureee.EQ1} in the most direct way, if possible.
 To this end, it seems appropriate to define $\Gamma$ by the inner $\eta$-integral in \eqref{distriheureee.EQ1}, that is
\begin{equation} \label{sec.one:mainThm_defGamma}
    \Gamma(z;\zeta) := \int_{\R^p}\widetilde{\Gamma}\Big((z,0);(\zeta,\eta)\Big)\,\d\eta \qquad (\text{for $z\neq \zeta$ in $\R^{1+n}$}).
\end{equation}
 If in \eqref{distriheureee.EQ1} we were allowed to take as a test function $\psi$ any function of the form $\varphi(z)$ in $C_0^\infty(\R^{1+n})$,
 then \eqref{distriheureee.EQ1} would easily prove that $\Gamma$ is a fundamental solution of $P$,
 in view of the fact that $\widetilde{P}(\varphi\circ \pi)=P\varphi$.
 Unfortunately,
 a test function $\varphi(z)$ on $\R^{1+n}$
 does not become a test function $\psi$ on $\R^{1+n+p}$ by simply con\-si\-de\-ring
 $\psi=\varphi\circ \pi$ (where $\pi$
 is the projection in \eqref{sec.one:def_eqLliftok}).

 A more promising procedure (still based on \eqref{distriheureee.EQ1}) is the ``product-like'' choice
 $$\psi(z,\xi)=\varphi(z)\,\theta_j(\xi),$$
 where $\theta_j\in C_0^\infty(\R_\xi^p)$ is such that
 $\theta_j\to 1$ as $j\to\infty$: indeed, one may formally let $j\to\infty$
  in the following identity (resulting from \eqref{distriheureee.EQ1} with this choice of $\psi$)
\begin{equation}\label{passtothelimit}
 \int_{\R^{1+n}}\d \zeta \int_{\R^p}\d \eta\,\, \widetilde{\Gamma}\Big((z,0);(\zeta,\eta)\Big)\,\widetilde{P}^*\big(\varphi(\zeta)\,\theta_j(\eta)\big)
  = -\varphi(z)\,\theta_j(0),
\end{equation}
  with the hope that, when $j\to\infty$ (by again exploiting the fact that $\widetilde{P}$ lifts $P$), this may lead to
  $$\int_{\R^{1+n}} \bigg(\int_{\R^p}\widetilde{\Gamma}\Big((z,0);(\zeta,\eta)\Big)\,\d \eta\bigg)\,P^* \varphi(\zeta)\,\d \zeta= -\varphi(z).$$
  In the end, the latter identity would produce the fact that the function $\Gamma$ in \eqref{sec.one:mainThm_defGamma}
  is indeed a global fundamental solution for $P$.\medskip

 In order to make this argument more than heuristic, it appears that some a priori
 assumptions must be conveniently made, namely:
\begin{itemize}
  \item we need to know that $\Gamma$ in \eqref{sec.one:mainThm_defGamma} is well posed as a convergent integral;
  we also need to know some summability properties of $\Gamma$ (implicit in the definition of a fundamental solution, see Section \ref{sec.existGamma});

  \item
  some structural and growth assumptions
 on the formal adjoint of the ``remainder'' operator $R:=\widetilde{P}-P$ (which operates on the lifting variables $\xi$ only)
 should be conveniently made to rigorously pass to the limit in
 \eqref{passtothelimit}.
\end{itemize}
 This discussion fully motivates the technical assumptions that we shall make in the saturation Theorem \ref{exTheoremA}, postponed to the next section.\medskip

 It is now time to describe in details the assumptions on the vector fields $X_j$
 in \eqref{Hiniziale}.
  Let $X=\{X_1,\ldots,X_m\}$
  be a set of smooth and linearly independent\footnote{The linear
  independence of $X_1,\ldots,X_m$ is meant in the vector space $\mathcal{X}(\R^n)$
  of the smooth vector fields on $\R^n$, and it must not be confused
  with the linear independence of the \emph{(tangent) vectors} $X_1(x),\ldots,X_m(x)$;
  for example, the Grushin vector fields in $\R^2$ defined by $X_1=\de_{x_1}$ and
  $X_2=x_1\de_{x_2}$ are linearly independent in $\mathcal{X}(\R^2)$, despite
  the vectors of $\R^2$ given by $X_1(x)\equiv (1,0)$ and $X_2(x)\equiv (0,x_1)$
  are dependent when $x_1=0$.}
  vector fields on $\R^n$ satisfying the following
  assumptions:
  \begin{itemize}
    \item[(H.1)]
 there exists a family of
  (non-isotropic) dilations $\{\delta_{\lambda}\}_{\lambda > 0}$
  of the form
 $$\dela:\R^n\longto\R^n, \qquad \dela(x) = (\lambda^{\sigma_1}x_1,\ldots,
 \lambda^{\sigma_n}x_n), $$
 where $1 = \sigma_1\leq \ldots \leq \sigma_n$ are real numbers, such that
   $X_1,\ldots,X_m$ are $\delta_\lambda$-homogeneous of degree
  $1$, i.e.,
  $$X_j(f\circ \dela)=\lambda\,(X_jf)\circ \dela,\quad
  \forall\,\,\lambda>0,\,\,f\in C^\infty(\R^n),\,\, j=1,\ldots,m; $$
     \item[(H.2)]
  the set $X$ satisfies H\"ormander's rank
  condition at $0$, i.e.,
  $$\dim\big\{Y(0) : Y \in \mathrm{Lie}\{X\}\big\} = n.$$
  \end{itemize}
 By $\mathrm{Lie}\{X\}$ we mean the smallest Lie sub-algebra
 of the smooth vector fields $\mathcal{X}(\R^n)$ on $\R^n$ containing $X$.
 Here $\mathcal{X}(\R^n)$ is equipped with its obvious structures of vector space and of Lie algebra.
  \begin{remark} \label{sec.two_1:remHormander}
   It is not difficult to show that, since $X_1,\ldots,X_m$
   are $\dela$-ho\-mo\-ge\-ne\-ous of degree $1$, the validity
   of H\"ormander's rank condition at $0$ implies the validity of the latter at any $x\in\R^n$, and that
   $n\leq \mathrm{dim}(\mathrm{Lie}\{X\})<\infty$.

  Thus, the H\"ormander parabolic operator
    $\Ht$ in \eqref{Hiniziale} is $C^\infty$-hypoelliptic on every open subset
    of $\R^{1+n}$.
    Moreover, $\Ht$ satisfies the Weak Maximum Principle
    on every bounded open subset of $\R^{1+n}$: this follows
    from (H.1)-(H.2), as is proved in \cite[Sect.\,8.4]{BiagiBonfBook}.
  \end{remark}

  The following
  result is relevant for our purposes, and it can be proved starting from 
  \cite[Thm.\,1.4]{BiagiBonfiglioliCCM} and \cite[Thm\,3.1]{BiagiBonfiglioliLast}. We refer to
  \cite[\S 1.4]{BLUlibro}
  for the
  notions of sub-Laplacian and of \emph{homogeneous\footnote{Essentially, this is a triple $(\RN,\star,D_\lambda)$ of a Lie group
  $(\RN,\star)$ and a family of dilations $D_\lambda$ which are group automorphisms.} Carnot group on $\RN$}.
  \begin{theorem}\label{exTheoremB}
   Assume that $X = \{X_1,\ldots,X_m\}$ 
   satisfies the above assumptions \emph{(H.1)} and \emph{(H.2)}.
   Moreover, let $N=\mathrm{dim}(\mathrm{Lie}\{X\})$.
   Then the following facts hold:\medskip

 \emph{1.} If $N=n$, there exists a homogeneous Carnot group
   $\G$ \emph{(}with underlying manifold $\R^n$ and the same dilations $\dela$ as in
   \emph{(H.1)}\emph{)} such that
   $X$ is a system of
   Lie-generators of $\mathrm{Lie}(\G)$; hence
   $\LL:=\sum\nolimits_{j = 1}^m X_j^2$
   is a sub-Laplacian on $\G$.\medskip

 \emph{2.} If $N>n$, there exist a homogeneous Carnot group
   $\G$ \emph{(}with underlying manifold $\R^N$\emph{)} and a system $\{Z_1,\ldots,Z_m\}$ of
   Lie-generators of $\mathrm{Lie}(\G)$ such
   that  $Z_i$ is a lifting of $X_i$ for every $i = 1,\ldots,m$
   \emph{(}in the previously defined sense\emph{)};
   hence the sub-Laplacian $\sum_{j = 1}^mZ_j^2$
   is a lifting of $\LL$.
%
\end{theorem}
 The demonstration of Theorem \ref{exTheoremB} is quite delicate:
  for example, the proof of (2) makes use of the \emph{global} lifting method for homogeneous
  vector fields proved by Folland \cite{Folland2}, a notable refinement
  of the local lifting technique introduced by Rothschild and Stein in \cite{RothschildStein}
  for H\"ormander PDOs: a proof of (2) can be found in \cite{BiagiBonfiglioliLast}.
  As for assertion (1) in Theorem \ref{exTheoremB}, one argues as follows: 
\begin{remark}\label{rem.sulladimenzoinen}
 Consider the following facts: 
\begin{itemize}
  \item $\mathrm{Lie}\{X\}$ is an $n$-dimensional Lie algebra of analytic vector fields in $\R^n$ (analyticity follows from the fact that the $X_j$'s have
  polynomial component functions, due to (H.1));

  \item $X$ is a H\"ormander system, due to (H.1)-(H.2) (see Remark \ref{sec.two_1:remHormander});

  \item any vector field $Y\in \mathrm{Lie}\{X\}$ is complete, i.e., the integral curves of $Y$ are defined on the whole of $\R$
  (this can be proved as a consequence of (H.1)).
\end{itemize}
 Under these three conditions, Theorem 1.4 in \cite{BiagiBonfiglioliCCM} proves that $\mathrm{Lie}\{X\}$ coincides with the Lie algebra
 of a Lie group $\G$ on $\R^n$. As a matter of fact, under assumption (H.1), this Lie group $\G$ turns out to be a homogeneous Carnot group
 with dilations $\dela$ (see e.g., \cite[Chapter 16]{BiagiBonfBook}). Thus (1) follows.
\end{remark}

  All this being said, our aim in this paper is to prove that a saturation/lifting approach
  can be performed for the Heat type operators $\Ht=\sum_{j=1}^m X_j^2-\de_t$, where
  $X_1,\ldots,X_m$ satisfy (H.1) and (H.2). To this end, it is enough to assume that
  $N>n$, since (by Theorem \ref{exTheoremB}-(1)) the case $N=n$ is already known (see Folland, \cite{Folland}).
  When $N>n$ we will obtain the existence of a global fundamental
  solution (also called Heat kernel) $\Gamma$ for $\Ht$ obtained via the saturation formula \eqref{sec.one:mainThm_defGamma},
  taking in this case the following special form
\begin{equation*}
  \Gamma(t,x;s,y) := \int_{\R^p}\Gamma_\G(t,x,0;s,y,\eta)\,\d\eta,
\end{equation*}
 where $\Gamma_\G$ is a fundamental solution
 for the Heat-type operator
 $$\mathcal{H}_\G:=\sum_{j=1}^m Z_j^2-\frac{\de}{\de t}$$
 on the Lie group $\R\times \G$ (here the Carnot group $\G$ and $Z_1,\ldots,Z_m$ are the same as in Theorem \ref{exTheoremB}).
 The existence of $\Gamma_\G$ was proved in \cite{Folland} (see also \cite{BLUpaper}), where
 it was also shown that it takes a group-convolution form; this will lead to the
 even more profitable expression
\begin{equation}\label{defiGammaesplicitinaPRE2}
  \Gamma(t,x;s,y) =\int_{\R^p}\gamma_\G\Big(s-t,(x,0)^{-1}\star(y,\eta)\Big)\,\d\eta,
\end{equation}
   where $\gamma_\G$ is the fundamental solution of  $\mathcal{H}_\G$
  with pole at the origin, and $\star$ is the group law of the Carnot group in Theorem \ref{exTheoremB}-(2).
  In showing that $\Ht$ satisfies the assumptions for the saturation procedure heuristically described above,
  one must also use the global Gaussian estimates of $\gamma_\G$
  (see e.g., Jerison, S\'anchez-Calle \cite{JerisonSanchezCalle};
  Kusuoka, Stroock \cite{KusuokaStroock1,KusuokaStroock2};
  Varopoulos, Saloff-Coste, Coulhon
  \cite{VaropoulosSaloffCosteCoulhon}).\medskip

  Strictly speaking, formula \eqref{defiGammaesplicitinaPRE2}
  does not equip $\Gamma$ with a tran\-sla\-tion\--in\-va\-rian\-ce property, as is shown
  by the Grushin-type example (see e.g., \cite{CalinChangFurutaniIwasaki})
  $$G=\Big(\frac{\de}{\de{x_1}}\Big)^2+\Big(x_1\,\frac{\de}{\de{x_2}}\Big)^2-\frac{\de}{\de t}.$$
  Nonetheless, \eqref{defiGammaesplicitinaPRE2} is a nicely ``hybrid'' expression of
  the fundamental solution of $\Ht$ as an integral of a translation-invariant
  kernel; this expression is indeed worthwhile since we shall derive from it
  plenty of properties of $\Gamma$, as is shown in the following theorem, our main result:
\begin{theorem}[\textbf{Existence and properties of the global Heat-kernels for homogeneous H\"ormander PDOs}]\label{mainteo}
 Let $X$ be a set of smooth vector fields on $\R^n$ satisfying assumptions \emph{(H.1)} and \emph{(H.2)}, and let
 us assume that
 $$N=\mathrm{dim}(\mathrm{Lie}\{X\})>n.$$
  Let $\mathcal{H}$ be the Heat-type operator on $\R^{1+n}$ defined in
  \eqref{Hiniziale}, and let us denote by
  $$\text{$z=(t,x)$ the points of $\R^{1+n}=\R_t\times \R^n_x$}.$$
  Then $\mathcal{H}$ admits a global fundamental solution
  $\Gamma(z;\zeta)$; this means that $\Gamma(z;\zeta)$
  is defined for any $z,\zeta\in \R^{1+n}$ and it satisfies the
  following property: for any $z \in \R^{1+n}$ \emph{(}the pole\emph{)},
  $\Gamma(z;\cdot)$ is in $L^1_{\loc}(\R^{1+n})$ and
\begin{equation*}
    \int_{\R^{1+n}}
    \Gamma(z;\zeta)\,\mathcal{H}^*\varphi(\zeta)\,\d \zeta = -\varphi(z),
    \qquad \text{for every $\varphi\in C_0^{\infty}(\R^{1+n})$},
    \end{equation*}
 where  $\Ht^* = \sum_j X_j^2+\de/\de t$ is the formal adjoint of $\mathcal{H}=\sum_j X_j^2-\de/\de t$.

 More precisely, we take as $\Gamma$ the integral function
\begin{equation}\label{defiGammaesplicitinaPRE1}
  \Gamma(z;\zeta)
  = \Gamma(t,x;s,y) = \int_{\R^p}\gamma_\G\Big(s-t,(x,0)^{-1}\star(y,\eta)\Big)\,\d\eta,
\end{equation}
 where $\gamma_\G$ is the unique
 fundamental solution, with pole at $0$ and vanishing at infinity,
 of the Heat-type operator $\mathcal{H}_\G:=\sum_{j=1}^m Z_j^2-\de/\de t$
 on $\R\times \G$ \emph{(}which is a lifting of $\Ht$\emph{)};
 the Carnot group $\G = (\RN,\star)$
 and the vector fields $Z_1,\ldots,Z_m$ are as in Theorem \ref{exTheoremB}-(2).
 The existence of $\gamma_\G$ is granted by \cite{Folland}.\medskip

 Moreover, $\Gamma$ in \eqref{defiGammaesplicitinaPRE1} also enjoys the following list of properties:
  \begin{itemize}
   \item[{(i)}] $\Gamma\geq 0$ and we have
   $$\text{$\Gamma(t,x;s,y) = 0$ if and only if $s\leq t$}.$$

   \item[{(ii)}]
   We have $\Gamma(t,x;s,y)=\Gamma(-s,x;-t,y)$, and
   $\Gamma$ depends on $t$ and $s$ on\-ly
  through $s-t$:
  $$\Gamma(t,x;s,y) = \Gamma(0,x;s-t,y)=\Gamma(t-s,x;0,y).$$
  Furthermore $\Gamma$ is symmetric in the space
  variables $x$ and $y$, i.e.,
  $$\Gamma(t,x;s,y) = \Gamma(t,y;s,x).$$

   \item[{(iii)}] For every $\lambda > 0$ we have
   $$\Gamma\Big(\lambda^2t,\dela(x);\lambda^2s,\dela(y)\Big) =
   \lambda^{-q}\,\Gamma(t,x;s,y),\quad \text{where $q = \txt\sum_{j = 1}^m \sigma_j$.}$$

   \item[{(iv)}] $\Gamma$ is smooth out of the diagonal
   of $\R^{1+n}\times\R^{1+n}$.

  \item[{(v)}] For every compact set
     $K\subseteq\R^{1+n}$, we have
   $$
    \lim_{\|\zeta\|\to\infty}\bigg(\sup_{z\in K}\Gamma(z;\zeta)\bigg)
    = \lim_{\|\zeta\|\to\infty}\bigg(\sup_{z\in K}\Gamma(\zeta;z)\bigg) = 0.
   $$

   \item[{(vi)}] $\Gamma\in L^1_{\loc}(\R^{1+n}\times\R^{1+n})$
   and, for every fixed $z\in\R^{1+n}$, we have
   $$\Gamma(z;\cdot),\,\,\Gamma(\cdot;z)\in L^1_{\loc}(\R^{1+n}).$$

  \item[{(vii)}] For every
  fixed $(t,x)\in\R^{1+n}$ we have
\begin{equation*}
   \int_{\R^n}\Gamma(t,x;s,y)\,\d y = 1, \quad \text{for every $s > t$}.
\end{equation*}

  \item[{(viii)}] For every fixed
  $\varphi\in C^\infty_0(\R^{1+n})$, the map
   defined by the potential function
  $$
  \R^{1+n}\ni \zeta\mapsto 
  \Lambda_\varphi(\zeta) := \int_{\R^{1+n}}\Gamma(z;\zeta)
  \,\varphi(z)\,\d z$$
  is smooth, it vanishes at infinity and $\Ht(\Lambda_\varphi) = -\varphi$
  on $\R^{1+n}$.

  \item[{(ix)}]
  If $\varphi\in C(\R^n)$ is bounded, then the potential-type function
  $$u(t,x) := \int_{\R^n}\Gamma(0,y;t,x)\,\varphi(y)\,\d y$$
  defined for $(t,x)\in \Omega=(0,\infty)\times \R^n$
  is the unique bounded classical solution of the ho\-mo\-ge\-neo\-us
  Cauchy problem
     \begin{equation*}
    \begin{cases}
     \Ht u = 0 & \text{in $\Omega$} \\
     u(0,x) = \varphi(x) & \text{for $x\in\R^n$.}
    \end{cases}
   \end{equation*}

   \item[{(x)}]
   For every $x,y\in\R^n$ and every $s,t > 0$, we have the reproduction formula
 \begin{equation*}
  \Gamma(0,y;t+s,x) = \int_{\R^n}\Gamma(0,w;t,x)\,\Gamma(0,y;s,w)\,\d w.
 \end{equation*}

  \end{itemize}
  Finally, if we consider the function $\Gamma^*$ defined by
  $$\Gamma^*(t,x;s,y) := \Gamma(s,y;t,x), \quad
  \text{for every $(t,x),(s,y)\in\R^{1+n}$},$$
  then $\Gamma^*$ is a global fundamental solution for
  $\Ht^* =\sum_{j=1}^m X_j^2+\de/\de_t$, satisfying dual
  statements of \emph{(i)}-through-\emph{(x)}.
\end{theorem}

 We observe that there exists at most one fundamental solution $\Gamma$
 of $\Ht$ such that, for any fixed $z\in \R^{1+n}$, it holds that
 $\Gamma(z;\cdot)$ is continuous out of $z$, and $\lim_{\|\zeta\|\to \infty}\Gamma(z;\zeta)=0$ (see Remark \ref{rem.onFSgeneral}-(c)).
 As a consequence (see properties (iv,v) above) the function $\Gamma$ satisfying the properties of Theorem  \ref{mainteo}
 is unique.

\begin{remark}
 Many of the properties (i)-to-(x), albeit not unexpected,
 are based on quite technical arguments made possible
 by the very formula \eqref{defiGammaesplicitinaPRE1}, which therefore
 proves to be fruitful even if it does not furnish
 Gaussian estimates for $\Gamma$ in a simple way.
 From a recent investigation with Marco Bramanti \cite{BiBoBra}, it appears that, in the case of the stationary operator
 $\LL=\sum_{j=1}^m X_j^2$, one can pass from the integral representation analogous to
 \eqref{defiGammaesplicitinaPRE1} to pointwise estimates of the fundamental solution (and of its derivatives)
 in terms of the Carnot-Carath\'eodory
 distance associated with $X_1,\ldots,X_m$: this requires some
 work, also based on results by Nagel, Stein, Wainger \cite{NSW},
 by S\'anchez-Calle \cite{SanchezCalle}, and
 by Bramanti, Brandolini, Manfredini, Pedroni \cite{BBMP}.
 We plan to return to the non-trivial problem of the pointwise Gaussian estimates of the Heat kernel $\Gamma$ in a future investigation.

 We also point out that the techniques of this paper can be used
 in order to obtain \emph{uniform} and global estimates for the
 fundamental solutions of the operators $\sum_{i,j} a_{i,j} X_iX_j-\de/\de t$, as
 the matrix $(a_{i,j})$ ranges over the $m\times m$ symmetric and positive-definite matrices
 satisfying a suitable (uniform) el\-lip\-ti\-ci\-ty condition (see also \cite{BLUpaper}).
 In its turn, we plan to use these uniform estimates to study the
 parametrices for non-constant $a_{i,j}$'s (see also \cite{BLUtrans}).
\end{remark}

 Our integral representation is also sufficiently helpful that it produces analogous
 representations for any higher order derivative, as this theorem shows:
\begin{theorem}[\textbf{Representation of the derivatives of $\Gamma$}]\label{teomainderiv}
 Let the as\-sump\-tions of Theorem \ref{mainteo} hold \emph{(}from which we inherit the notation\emph{)}, and   
 let $\Gamma$ be the fundamental solution of $\Ht$
 in \eqref{defiGammaesplicitinaPRE1}.

  Then, for any
  $\alpha,\beta\in\N\cup\{0\}$, any $h,k\geq 1$ and any
  choice of indexes $i_1,\ldots,i_h,j_1,\ldots,j_k$ in $\{1,\ldots,m\}$, we have the following
  representation formulas \emph{(}holding true for $(t,x)\neq
  (s,y)$ in $\R^{1+n}$\emph{)}, respectively concerning $X$-derivatives in the $y$-variable,
  in the $x$-variable, and in the mixed $(x,y)$-case:
\begin{align}
  & \Big(\frac{\de}{\de s}\Big)^\alpha
  \Big(\frac{\de}{\de t}\Big)^\beta
  X^y_{i_1}\cdots X^y_{i_h}\Gamma(t,x;s,y) \label{eq.derGammays} \\[0.15cm]
  & \quad =
       (-1)^{\beta }\int_{\R^p}
   \bigg(
   \Big(\frac{\de}{\de s}\Big)^\alpha
  \Big(\frac{\de}{\de t}\Big)^\beta {Z}_{i_1}\cdots {Z}_{i_s} \gamma_\G \bigg)\Big(s-t,
   (x,0)^{-1}\star (y,\eta)\Big) \,\d\eta; \nonumber \\[0.3cm]
   & \Big(\frac{\de}{\de s}\Big)^\alpha
  \Big(\frac{\de}{\de t}\Big)^\beta
  X^x_{j_1}\cdots X^x_{j_k}\Gamma(t,x;s,y) \label{eq.derGammatx} \\[0.15cm]
   & \quad =
   (-1)^{\beta }\,\int_{\R^p}
   \bigg(\Big(\frac{\de}{\de s}\Big)^\alpha   \Big(\frac{\de}{\de t}\Big)^\beta
   {Z}_{j_1}\cdots {Z}_{j_k}\gamma_\G \bigg)
   \Big(s-t,(y,0)^{-1}
   \star (x,\eta)\Big)\,
   \d\eta; \nonumber \\[0.3cm]
   & \Big(\frac{\de}{\de s}\Big)^\alpha\,
   \Big(\frac{\de}{\de t}\Big)^\beta\,X^x_{j_1}\cdots X^x_{j_k}X^y_{i_1}\cdots X^y_{i_h}
   \Gamma(t,x;s,y) \label{eq.derGammatutte}  \\[0.15cm]
   &\qquad =
   (-1)^{\beta }\,\int_{\R^p}
   \bigg(\Big(\frac{\de}{\de s}\Big)^\alpha \Big(\frac{\de}{\de t}\Big)^\beta
   {Z}_{j_1}\cdots {Z}_{j_k}
   \Big(\big({Z}_{i_1}
   \cdots {Z}_{i_h}\gamma_\G \big)\circ\widetilde{\iota}\Big)\bigg) \nonumber \\
   & \qquad\qquad\qquad\qquad\qquad\qquad\qquad\qquad\qquad\,\,\quad
   \Big(s-t,(y,0)^{-1}\star (x,\eta)\Big)\,\d\eta\,.\nonumber
\end{align}
 Here $\widetilde{\iota}:\R^{1+N}\to\R^{1+N}$ is the map
 defined by
 $$\widetilde{\iota}(t,(x,\xi)) = (t,(x,\xi)^{-1})\qquad
 (\text{with $t\in\R$, $x\in\R^n$, $\xi\in\R^p$}),$$
 and $(x,\xi)^{-1}$ is the inverse
 of $(x,\xi)$ in the Lie group $\G = (\RN,\star)$; moreover,
 ${Z}_1,\ldots,{Z}_m$ are the lifting vector fields of
 ${X}_1,\ldots,{X}_m$ as in Theorem \ref{exTheoremB}.\bigskip
 \end{theorem}

  The plan of the paper is now in order:
  \begin{enumerate}[-]
   \item in Section \ref{sec.existGamma}
   we use Theorem \ref{exTheoremB} to prove the existence
   of $\Gamma$ as in Theorem \ref{mainteo};

\item
 in Section \ref{sec.DERIVPbHT} we prove
 Theorem \ref{teomainderiv}, furnishing the integral representation
 of the higher order derivatives of $\Gamma$;

 \item in Section \ref{sec.CauchyPbHT} we briefly
  study the existence and the uniqueness of the
  solutions of the Cauchy problem for $\Ht$;


     \item in Section \ref{sec.propertGamma} we prove all the distinguished features
   of $\Gamma$ in Theorem \ref{mainteo}.
  \end{enumerate}

  \section{Existence of a global fundamental solution for $\mathcal{H}$}
   \label{sec.existGamma}
   In the sequel, we tacitly inherit all the notations and assumptions in Theorem \ref{mainteo}.
   In this section
   we shall prove the existence of a global fundamental
   solution for $\mathcal{H}$.
  To begin with, for the sake of clarity, we remind the definition
   of a (global) fundamental solution for a generic
   smooth linear PDO $P$.
   \begin{definition} \label{sec.one:defi_Green}
    On Euclidean space $\RN$, we consider a linear PDO
    $$P = \sum_{|\alpha|\leq d}a_{\alpha}(x)\,D^\alpha_x,$$
    with smooth real-valued coefficients $a_\alpha(x)$ on $\RN$. We say that
    a function
    $$\Gamma:\{(x,y)\in\RN\times\RN: x\neq y\}\longto\R,$$
    is a (global) fundamental solution for $P$ if it satisfies the
    following property: for every $x \in \R^n$, the
     function $\Gamma(x;\cdot)$ is locally integrable
     on $\RN$ and
    \begin{equation} \label{sec.two:eq_GreenGamma}
    \int_{\RN}\Gamma(x;y)\,P^*\varphi(y)\,\d y = -\varphi(x)
    \qquad \text{for every $\varphi\in C_0^{\infty}(\R^N,\R)$},
    \end{equation}
     where $P^*$ denotes the formal adjoint of $P$.
   \end{definition}
   \begin{remark} \label{rem.onFSgeneral}
  (a)\,\, The existence of a global fundamental solution
  for $P$ is far from being obvious and it is, in general,
  a very delicate issue.
  In the particular case of $C^{\infty}$-hypoelliptic linear PDOs $P$
  having a $C^{\infty}$-hypoelliptic formal adjoint $P^*$,
  it is possible to prove the \emph{local}
  existence of a fundamental solution on a suitable neighborhood of
  each point of $\R^N$ (see, e.g., \cite{Treves}; see also \cite{Bony}).\medskip

  (b)\,\, Fundamental solutions are, in general, not unique
  since the addition of a $P$-harmonic function (that is, a smooth function
  $h$ such that $P h=0$ in $\RN$) to a fundamental solution produces another
  fundamental solution. \medskip

  (c)\,\, Nonetheless, if  $P$ is $C^\infty$-hypoelliptic and fulfills
  the Weak Maximum Principle on every bounded open set of $\RN$, then
  there exists at most one fundamental solution $\Gamma$ for $P$ such that
  $$\lim_{\|y\|\to\infty}\Gamma(x;y) = 0, \quad \text{for every $x\in\RN$}.$$
  Indeed, if $\Gamma_1,\Gamma_2$ are two such
  functions, then (for every fixed $x\in\RN$) the map
  $u_x := \Gamma_1(x,\cdot)-\Gamma_2(x,\cdot)$
  belongs to $L^1_{\mathrm{loc}}(\R^N)$
  and it is a solution of $P u_x = 0$ in the weak sense of distributions on $\RN$;
  the hypoellipticity of $P$ ensures that $u_x$ is (a.e.\,equal to) a
  smooth function on $\RN$
  which vanishes at infinity by the assumptions on $\Gamma_1,\Gamma_2$;
  from the Weak Maximum Principle for $P$
  it is standard to obtain that  $\Gamma_1\equiv \Gamma_2$ (a.e.).
 \end{remark}

   Next, as explained in Section \ref{sec:intro},
   we need the following theorem.
  Despite its seemingly technical assumptions,
 this theorem is applicable in many interesting situations,
 as we shall discuss in Example \ref{exam.theoAAAA}.
 \begin{theorem}[See \protect{\cite[Theorem 2.5]{BiagiBonfiglioliLast}}]\label{exTheoremA}
 Let $P$ be a smooth linear PDO on $\R_z^N$,
 and let $\widetilde{P}$ be a lifting of $P$ on
 $\R_z^N\times\R_\xi^p$ which satisfies
 the following structural assumptions:\medskip

  \noindent \emph{(S.1):}
    the formal adjoint $R^*$ of $R := \widetilde{P}-P$
    annihilates any $u\in C^2(\R_z^N\times \R_\xi^p)$ independent of $\xi$, i.e.,
  \begin{equation}\label{sec.one:rem_eqformPstar}
   R^*=\sum_{\beta \neq 0}
   r^*_{\alpha,\beta}(z,\xi)\,\bigg(\frac{\de}{\de z}\bigg)^\alpha \bigg(\frac{\de}{\de\xi}\bigg)^\beta,
  \end{equation}
 for \emph{(}finitely many, possibly identically vanishing\emph{)}
 smooth functions $r^*_{\alpha,\beta}(z,\xi)$;\medskip

 \noindent \emph{(S.2):}
 there exists a
 sequence $\{\theta_j(\xi)\}_j$ in $C_0^{\infty}(\R^p,[0,1])$
 such that\footnote{By this we mean that,
 denoting by $\Omega_j$
 the set $\{\xi\in \R^p:\theta_j(\xi)=1\}$, one has
 $$\bigcup_{j\in \mathbb{N}} \Omega_j=\R^p\quad \text{and}\quad
 \Omega_j\subset \Omega_{j+1}\quad
 \text{for any $j\in\mathbb{N}$}.$$}
 $$\text{$\{\theta_j=1\}\uparrow \R^p$ as $j\uparrow \infty$},$$
 with the following property:
 for every compact set $K\subset \R^n$
 and for any  coef\-fi\-cient function $r^*_{\alpha,\beta}$ of $R^*$ as in
 \eqref{sec.one:rem_eqformPstar} one can find constants $C_{\alpha,\beta}(K)$
 s.t.
\begin{equation*}
   \Big|r^*_{\alpha,\beta}(z,\xi)
   \Big(\frac{\de}{\de\xi}\Big)^\beta\theta_j(\xi)\Big|
   \leq C_{\alpha,\beta}(K),
 \end{equation*}
 uniformly for every $z\in K$, $\xi\in \R^p$ and $j \in \mathbb{N}$.\medskip

  Assume that $\widetilde{P}$ admits a global fundamental solution
 $\widetilde{\Gamma}=\widetilde{\Gamma}\big((z,\xi);(\zeta,\eta)\big)$
 \emph{(}with pole $(z,\xi)$\emph{)}
  satisfying the following integrability assumptions:
 \begin{itemize}
   \item[\emph{(i)}] for every fixed $z, \zeta \in \R^N$ with $z \neq \zeta$, it holds that
   \begin{equation*}
   \eta \mapsto \widetilde{\Gamma}\big((z,0);(\zeta,\eta)\big) \quad
   \text{belongs to $L^1(\R^p)$},
   \end{equation*}

   \item[\emph{(ii)}] for every fixed $z \in \R^N$ and every compact set
   $K\subseteq\R^N$, it holds that
   \begin{equation*}
   (\zeta,\eta)\mapsto
   \widetilde{\Gamma}\big((z,0);(\zeta,\eta)\big)\quad
   \text{belongs to $L^1(K\times\R^p)$}.
   \end{equation*}
  \end{itemize}

  Then the function $\Gamma$  defined by \eqref{sec.one:mainThm_defGamma}
   is a global fundamental solution for $P$ on  $\R^N$ with pole $z$.
\end{theorem}
\begin{example}\label{exam.theoAAAA}
 Theorem \ref{exTheoremA} can be applied in the following examples:\medskip

 1) The choices of lifting pairs $(P,\widetilde{P})$
 given by
 $$\text{$(\Delta_n, \Delta_{n+p})$ and $(\Ht_n, \Ht_{n+p})$}$$
 trivially satisfy assumptions (S.1)-(S.2) and (i)-(ii)
 of Theorem \ref{exTheoremA}.\medskip

 2) A less trivial example is given (as a very particular case of the PDOs in the present paper) by the ``parabolic Grushin operator'' on $\R^3_z\equiv \R_t\times \R^2_x$
  (where $z=(t,x)$) i.e.,
 $$G=\frac{\de^2}{\de x_1^2}+x_1^2\,\frac{\de^2}{\de x_2^2}-\frac{\de}{\de t},$$
 with a lifting given by
 $$\widetilde{G}=\frac{\de^2}{\de x_1^2}+\bigg(\frac{\de}{\de \xi}+x_1\,\frac{\de}{\de x_2}\bigg)^2
 -\frac{\de}{\de t}
 \quad \text{on $\R_t\times\R^2_x\times \R_\xi$.}$$
 As we shall see, for this last example
 not only (S.1)-(S.2) are satisfied,
 but there also exists a fundamental solution $\widetilde{\Gamma}$ for
 $\widetilde{G}$ satisfying hypotheses {(i)-(ii)}
 of Theorem \ref{exTheoremA}. Therefore, we can infer that $G$ admits a global fundamental solution given by the
 saturation function \eqref{sec.one:mainThm_defGamma}.\medskip

  3) More generally, in the paper
  \cite{BiagiBonfiglioliLast} a meaningful
  case is described where Theorem \ref{exTheoremA} can always be applied: namely,
  any H\"ormander sum of squares $P=\sum_{j=1}^m X_j^2$, where $X_1,\ldots,X_m$ satisfy
  axioms (H.1)-(H.2), fulfils the as\-sump\-tions of Theorem \ref{exTheoremA}, thus
   admitting a global fundamental solution.
\end{example}

  Now, we proceed as follows: first we use Theorem \ref{exTheoremB} to prove the existence of a lifting $\widetilde{\mathcal{H}}$
   for $\Ht$ satisfying assumptions (S.1) and (S.2) of Theorem \ref{exTheoremA};
   then we show the existence of a fundamental solution $\widetilde{\Gamma}$ for $\widetilde{\mathcal{H}}$
   fulfilling conditions (i) and (ii) of Theorem \ref{exTheoremA}: the latter will then ensure the
   existence of a fundamental solution $\Gamma$ for $\mathcal{H}$.\medskip

  According to Theorem \ref{exTheoremB},
  given a family $X$ of vector fields in $\R^n$ satisfying axioms (H.1)-(H.2), and setting
  $N=\mathrm{dim}(\mathrm{Lie}\{X\})$,
  it is possible to find
  a homogeneous Ca\-rnot group $\G = (\RN,\star,D_\lambda)$ on $\RN=\R_x^n\times\R_\xi^p$
  (with $m$ generators and nilpotent of step $r = \sigma_n$)
  and a system $\{Z_1,\ldots,Z_m\}$ of Lie-generators of
  $\LieG$ such that, for every $i = 1,\ldots, m$,
  $Z_i$ is a lifting of $X_i$.
  It can also be shown that the dilations $\{D_\lambda\}_{\lambda > 0}$ on $\G$
  take the form
  \begin{equation} \label{eq.dsplitdeladelastar}
   D_\lambda(x,\xi) = \big(\dela(x),\delta_\lambda^* (\xi)\big),
   \qquad \text{for every $(x,\xi)\in\RN=\R_x^n\times\R_\xi^p$},
   \end{equation}
   where $\delta_\lambda^* $ is another family of non-isotropic
   dilations on $\R^p$ which we write as
  \begin{equation} \label{eq.explicitdelastar}
   \delta_\lambda^* (\xi) = (\lambda^{\sigma_1^*}\xi_1,\ldots,
   \lambda^{\sigma_p^*}\xi_p), \quad
   \xi\in\R^p.
  \end{equation}
   Note that, at this stage, three homogeneous dimensions naturally arise:
\begin{equation} \label{delaprodottosaQQQQQQQQQQ}
 \txt q := \sum_{j = 1}^n\sigma_j,\quad  q^* := \sum_{j = 1}^p\sigma_j^*,\quad Q = q+q^*,
\end{equation}
  which are, respectively, the homogeneous dimensions of $(\R^n,\dela)$,
  $(\R^p,\delta_\lambda^*)$, $(\RN,D_\lambda)$. Accordingly,
  we fix the canonical homogeneous norms $S,N,h$ on
  the spaces $\R^n,\R^p,\R^N$ respectively,
  defined by
  \begin{equation} \label{sec.two_2:eq_defcanonicalh}
   S(x):=\sum_{j = 1}^n |x_j|^{1/\sigma_j},\quad N(\xi):=\sum_{j = 1}^p
   |\xi_j|^{1/\sigma_j^*},\qquad  h(x,\xi) := S(x)+N(\xi).
  \end{equation}
 We note that any homogeneous norm $d$ on $\G$ is
 controlled by $h$ (from above and below times suitable constants);
 see \cite[Proposition 5.1.4]{BLUlibro}.
\begin{remark} \label{sec.two_2:remChangeF}
   For strictly technical reasons, following \cite{BiagiBonfiglioliLast}, we need to look at
    the following \textquotedblleft con\-vo\-lu\-tion\--like'' map
   \begin{equation*}
   F:\R^n\times\R^n\times \R^p\longto\RN, \qquad
   F(x,y,\eta) := (x,0)^{-1}\star (y,\eta).
   \end{equation*}
   As in \cite[Chapter 1.3]{BLUlibro}), one can prove that
  \begin{equation} \label{sec.two_2:eq_expressionconvmap}
    \begin{split}
     F_1 (x, y,\eta) & = y_1-x_1, \\
     F_i (x,y,\eta) & = y_i-x_i+ p_i (x,y,\eta)  \qquad (i = 2,\ldots,n), \\
     F_{n+k} (x,y,\eta) & = \eta_k+ q_k (x,y,\eta),\qquad  (k = 1,\ldots,p),
    \end{split}
   \end{equation}
  where, $p_i$ and $q_k$ are
  polynomials with the following features:
\begin{itemize}
  \item[-]
  $p_i$ only depends on those variables
  $x_h,y_h$ and $\eta_j$ such that
  $\sigma_h,\sigma^*_j< \sigma_i$;

  \item[-] $q_k$ only depends on those variables
  $x_h,y_h$ and $\eta_j$ such that
  $\sigma_h,\sigma^*_j< \sigma^*_k$;

  \item[-] $p_i (0,y,\eta) = q_k (0,y,\eta ) = 0$,
  for every $(y,\eta)\in\RN$.
 \end{itemize}
 Let now $x,y\in\R^n$ be fixed.
 Since $q_1$  {does not depend} on $\eta_1,\ldots,\eta_p$ and since,
   for every $k \in\{2,\ldots,p\}$, $q_k$ only depends on $\eta_1,\ldots,\eta_{k-1}$, we see that
   the map
   \begin{equation} \label{sec.two_2:def_mapPsixy}
    \Psi_{x,y}: \R^p\longto\R^p, \quad \Psi_{x,y}(\eta) :=
    \Big(F_{n+1} (x,y,\eta),
    \ldots,F_{n+p} (x,y,\eta) \Big)
   \end{equation}
   defines a  {$C^\infty$-diffeomorphism} of $\R^p$,
   with polynomial components.
   Hence, in particular, $\Psi_{x,y}$ is a proper map,
   which is equivalent to saying that
   $$\lim_{\|\eta\|\to\infty}\|\Psi_{x,y}(\eta)\| = \infty.$$
   Furthermore, by \eqref{sec.two_2:eq_expressionconvmap}, one has
   \begin{equation*}
    \det (\mathcal{J}_{\displaystyle \Psi_{x,y}}(\eta) ) = 1,
    \quad \text{for every $\eta \in \R^p$}.
   \end{equation*}
   The map $\Psi_{x,y}$ will be repeatedly used as a change of variable
   in integral estimates; indeed, one has
\begin{equation*}
    (x,0)^{-1}\star \Big(y,\Psi^{-1}_{x,y}(\eta')\Big)=\bigg(F_1\Big(x,y,\Psi^{-1}_{x,y}(\eta')\Big),\ldots,F_n\Big(x,y,\Psi^{-1}_{x,y}(\eta')\Big),\eta' \bigg);
\end{equation*}
 consequently, with the notation in \eqref{sec.two_2:eq_defcanonicalh},
 $\Psi_{x,y}$ enjoys the nice (technical) feature
\begin{equation}\label{sservePsi}
 h\Big((x,0)^{-1}\star \Big(y,\Psi^{-1}_{x,y}(\eta')\Big)\Big)\geq N(\eta').
\end{equation}
 Here $h$ can be replaced by any homogeneous norm $d$ on $\G$, times some constant.
 \end{remark}
   If $\LL_\G=\sum_{j=1}^m Z_j^2$,
  it is straightforward to recognize that the Heat operator
  $\Ht_\G = \LL_\G - \de_t$ is a lifting of $\Ht=\LL-\de_t$ on $\R^{1+N} = \R_t\times\R_x^{n}\times\R_\xi^p$,
  that is,
  $$\Ht_\G(u\circ \pi)(t,x,\xi) = (\Ht u)(t,x),
  \quad \text{$\forall\,\,t\in\R$, $(x,\xi)\in\RN$, $u\in C^2(\R^{1+n})$},$$
  where $\pi:\R^{1+N}\to\R^{1+n}$ is the canonical projection
  of $\R^{1+N}$ onto $\R^{1+n}$.
  Our aim is now to prove that the operator $\Ht_\G$,
  as a lifting of $\Ht$, satisfies the assumptions
  (S.1) and (S.2) in Theorem \ref{exTheoremA}.
  \begin{lemma} \label{lem.HtGsaturable}
   The o\-pe\-rator $\Ht_\G$, as a lifting of $\Ht$,
   satisfies assumptions
  \emph{(S.1)-(S.2)} in Theorem \ref{exTheoremA}.
  \end{lemma}
  \begin{proof} (S.1):\,\,First of all we observe that, by definition, we have
   $$\text{$R := \Ht_\G - \Ht = \LL_\G - \LL$ on $\R^{1+N}$;}$$
   thus, since both $\LL_\G$ and $\LL$
   are self-adjoint (as they are sums of squares
   of homogeneous vector fields)
   we get $R^* = R$; moreover, as $\LL_\G$
   is a lifting of $\LL$, we infer that $R$ annihilates any $C^2$ function
   independent of $\xi$. \medskip

   (S.2):\,\,If $N$ is as in \eqref{sec.two_2:eq_defcanonicalh}, we choose a function $\theta\in C^\infty_0(\R^p,[0,1])$
   such that
 $$\mathrm{supp}(\theta) \subseteq \{\xi \in \R^p: N(\xi) \leq 2\};\qquad
  \text{$\theta \equiv 1$ on $\{\xi\in\R^p: N(\xi) < 1\}$.}$$
   We
   define a sequence $\{\theta_j\}_{j}$
   in $C^\infty_0(\R^p)$ by setting,
   for every $j\in \mathbb{N}$,
   \begin{equation*}
  \theta_j(\xi) := \theta (\delta^*_{2^{-j}}(\xi) ),
  \quad \text{for $\xi\in\R^p$}.
   \end{equation*}
   By arguing exactly as in \cite[Theorem 4.4]{BiagiBonfiglioliLast},
   after several technical
   computations (based on the homogeneity of the $Z_j$ and on the structure of $\delta_\lambda^*$) one can recognize that $\{\theta_j\}_{j}$
   satisfies the properties in assumption (S.2).
     \end{proof}

  With Lemma \ref{lem.HtGsaturable} at hand, the path towards the existence
  of a global fun\-da\-men\-tal solution for $\Ht$ is traced
  in Theorem \ref{exTheoremA},
  and it consists of two parts:
  \begin{enumerate}
   \item[(1)] firstly, we prove that
   $\Ht_\G$ admits a fundamental solution $\Gamma_\G$;

   \item[(2)] secondly, we show that such a
   $\Gamma_\G$ satisfies the integrability
   assumptions (i)-(ii) in Theorem \ref{exTheoremA}.
  \end{enumerate}
  As for (1), it follows from the first statement in the next result; in the sequel,
  in order to avoid the cumbersome
   notation $(t,(x,\xi))$ for the points
   in the product space $\R\times \R^N=\R_t\times (\R^n_x\times \R^p_\xi)$
   we often  write $(t,x,\xi)$.
  \begin{theorem}[\protect{\cite[Theorems 2.1, 2.5]{BLUpaper}}]\label{exTheoremC}
 There exists a map
  $$\gamma_\G:\R^{1+N}\equiv \R^{1+n+p}\to \R,$$
  smooth away from the origin,
   such that
   \begin{equation} \label{eq.defiGammaGHtG}
    \Gamma_\G(t,x,\xi;s,y,\eta) :=
    \gamma_\G\Big(s-t,(x,\xi)^{-1}\star(y,\eta)\Big)
   \end{equation}
   is a global fundamental solution of the operator
   $\Ht_\G=\LL_{\G}-\de_t$.
   In its turn, there exists a unique symmetric homogeneous norm on $\G$
    \emph{(}in the sense of \emph{\cite{BLUlibro}}\emph{)}
     $d\in C^\infty(\RN\setminus\{0\})$
    such that
    $$d^{2-Q}\big((x,\xi)^{-1}\star(y,\eta)\big), \qquad
    (x,\xi)\neq (y,\eta)$$
    is the global fundamental solution of $\LL_\G$
    \emph{(}where $Q$ is as in \eqref{delaprodottosaQQQQQQQQQQ}\emph{)}.
    The following Gaussian estimates for $\gamma_\G$ hold:
    there exists a constant $\mathbf{c} > 0$ such that,
    for every $(x,\xi)\in\RN$ and every $t > 0$, one has
    \begin{equation} \label{eq.Gaussest}
     \mathbf{c}^{-1}\,t^{-Q/2}\,\exp\Bigg(
     -\frac{\mathbf{c}\,d^2(x,\xi)}{t}\Bigg)
     \leq \gamma_\G(t,x,\xi) \leq
     \mathbf{c}\,t^{-Q/2}\,\exp\Bigg(
     -\frac{d^2(x,\xi)}{\mathbf{c}\,t}\Bigg).
    \end{equation}
 Via \eqref{eq.defiGammaGHtG}, global Gaussian estimates
 analogous to \eqref{eq.Gaussest} hold true for $\Gamma_\G$.\medskip

  Moreover, $\gamma_\G$ satisfies the following
   additional properties:
   \begin{itemize}
    \item[\emph{(i)}] $\gamma_\G\geq 0$ and $\gamma_\G(t,x,\xi) = 0$ if and only if $t \leq 0$;

    \item[\emph{(ii)}] $\gamma_\G(t,x,\xi) = \gamma_\G(t, (x,\xi)^{-1})$
    for every  $(t,x,\xi)$;

    \item[\emph{(iii)}] for every $\lambda > 0$ and
    every $(t,x,\xi)$, we have
    $$\gamma_\G\big(\lambda^2t, \delta_\lambda  x,\delta_\lambda^* \xi\big) =
    \lambda^{-Q}\,\gamma_\G(t,x,\xi),$$
    where $Q=q+q^*$ is the homogeneous dimension of the group $\G$;

    \item[\emph{(iv)}] $\gamma_\G$ vanishes at infinity, that is,
    $\gamma_\G(t,x,\xi) \to 0$ as $\|(t,x,\xi)\|\to\infty$;

    \item[\emph{(v)}] for every $t > 0$, we have
    $$\int_{\R^n\times \R^p}\gamma_\G(t,x,\xi)\,\d x\,\d \xi = 1.$$
   \end{itemize}
   Finally, if we consider the function
   $\Gamma_\G^*$ defined by
   \begin{equation*}
   \Gamma_\G^*(t,x,\xi;s,y,\eta) := \Gamma_\G(s,y,\eta;t,x,\xi),
   \end{equation*}
   then $\Gamma_\G^*$ is a global fundamental solution
   for the adjoint operator $\Ht_\G^* = \LL_\G+\de_t$.
\end{theorem}
  As for (2), the needed integrability
  properties of $\Gamma_\G$ rely on
  the Gaussian estimates of $\gamma_\G$ in \eqref{eq.Gaussest}, as we prove in the next result.
 \begin{theorem}\label{thm.Gammaint}
  Let the notation of Theorem \ref{exTheoremC}
  apply. Then the global fun\-da\-men\-tal solution
  $\Gamma_\G$ of $\Ht_\G$ satisfies the integrability
  assumptions
  \emph{(i)} and \emph{(ii)} in Theorem \ref{exTheoremA}.
 \end{theorem}
 \begin{proof}
   We first prove that $\Gamma_\G$
   satisfies assumption (i). According to
   Theorem \ref{exTheoremA},
   we have to show that,
   for fixed $(t,x)\neq(s,y)\in\R^{1+n}$, one has
   \begin{equation} \label{eq.toprovestati}
   \eta\mapsto\Gamma_\G(t,x,0;s,y,\eta)\in L^1(\R^p).
   \end{equation}
   If $s\leq t$, the above \eqref{eq.toprovestati}
   is an immediate consequence
   of Theorem \ref{exTheoremC}, since
   $$\Gamma_\G(t,x,0;s,y,\eta) \stackrel{\eqref{eq.defiGammaGHtG}}{=}
    \gamma_\G\Big(s-t,(x,0)^{-1}\star(y,\eta)\Big) = 0, \quad
    \text{for every $\eta\in\R^p$}.$$
    We can then assume that $s > t$. In this case,
    by \eqref{eq.Gaussest} and by performing
    the change of variables
    $\eta = \Psi_{x,y}^{-1}(u)$
    (see \eqref{sec.two_2:def_mapPsixy} in Remark \ref{sec.two_2:remChangeF}), we obtain the estimate
    \begin{equation*}
     \begin{split}
     & \int_{\R^p}\Gamma_\G(t,x,0;s,y,\eta)\,\d\eta \\
    & \quad \leq \frac{\mathbf{c}}{(s - t)^{Q/2}}\int_{\R^p}
    \exp\left(-\frac{d^2\big((x,0)^{-1}\star(y,\Psi_{x,y}^{-1}(u)\big)}
    {\mathbf{c}\,(s - t)}
    \right)\d u.
   \end{split}
  \end{equation*}
  On the other hand,
  since $d$ is a homogeneous norm on $\G$, it is possible to
  find a constant $\alpha=\alpha(\G)>0$ such that,
  for every $u\in\R^p$ and every $x,y\in \R^n$,
  \begin{equation*}
   \begin{split}
    d^2\big((x,0)^{-1}\star(y,\Psi_{x,y}^{-1}(u)\big) &\geq\alpha\,
  h^2\big((x,0)^{-1}\star(y,\Psi_{x,y}^{-1}(u)\big)\stackrel{\eqref{sservePsi}}{\geq}
  \alpha\,N^2(u),
   \end{split}
  \end{equation*}
  where $h,N$ are as in \eqref{sec.two_2:eq_defcanonicalh}.
 Hence,
 \eqref{eq.toprovestati} will follow if we show that
  \begin{equation} \label{eq.toprovestatibis}
   u\mapsto
  \varphi(u) :=
  \exp\left(-\frac{\alpha\,N^2(u)}{\mathbf{c}(s-t)}\right)\in L^1(\R^p).
  \end{equation}
  Now, since $\varphi\in C(\R^p)$, we obviously have
  $\varphi\in L^1_{\loc}(\R^p)$; moreover,
  since $\exp(-|r|) \leq \beta_Q\,(1+|r|)^{-Q/2}$ (for some constant $\beta_Q>0$), we get
  $$\varphi(u) \leq \frac{\beta_Q\,\big(\mathbf{c}\,(s-t)\big)^{Q/2}}
  {\big(\mathbf{c}(s-t) + \alpha\,N^2(u)\big)^{Q/2}}
  \leq \beta\,(s-t)^{Q/2}\,N^{-Q}(u),
  \quad \text{$\forall\,\,u\in\R^p\setminus\{0\}$}.$$
  We are then left to prove that $N^{-Q}$ is integrable
  away from $0$, namely on the set $\{N\geq 1\}$.
  This follows from $Q>q^*$ and by a standard diadic/homogeneous argument using the annuli
  $$C_n := \{u\in \R^p: 2^{n-1}\leq N(u) < 2^n\}.$$
  To complete the proof, we are left to show that
  $\Gamma_\G$ also satisfies (ii) in Theorem \ref{exTheoremA}: for any fixed $(t,x)\in \R^{1+n}$
  and any compact set $K\subseteq\R^{1+n}$, we prove 
  $$((s,y),\eta)\mapsto \Gamma_\G(t,x,0;s,y,\eta)\in L^1(K\times\R^p).$$
  Let $a,b$ be such that $K\subseteq[a,b]\times\R^n$.
  We have (see (i)-(v) in Theorem \ref{exTheoremC})
  \begin{align*}
   \int_{K\times\R^p}\Gamma_\G(t,x,0&; s,y,\eta)\,\d s\,\d y\,\d \eta
    \leq \int_{a}^{b}\left(\int_{\R^n\times \R^p}\Gamma_\G(t,x,0;s,y,\eta)\,
   \d y\,\d \eta\right)\d s \\
   & = \int_{a}^{b}\left(\int_{\RN}\gamma_\G\Big(s-t,(x,0)^{-1}\star(y,\eta)\Big)\,
   \d y\,\d \eta\right)\d s \\
   & \big(\text{by the change of variables $(y,\eta) = (x,0)\star(u,v)$}\big) \\
   & = \int_{a}^{b}\left(\int_{\RN}\gamma_\G(s-t,u,v)\,\d u\,\d v\right)
   \d s \leq \int_a^{b}1\,\d s = b-a,
  \end{align*}
  and the proof is complete.
 \end{proof}
 \begin{remark} \label{rem.crucialestim}
  It is contained in the proof of Theorem \ref{thm.Gammaint}
  the following fact: there exists a constant
  $\beta> 0$ such that,
  for every $(t,x),(s,y)\in\R^{1+n}$ with $s > t$ and
  for every $u\in\R^p\setminus\{0\}$, one has
\begin{equation} \label{eq.crucialestim}
 \Gamma_\G\Big(t,x,0;s,y,\Psi_{x,y}^{-1}(u)\Big)
  = \gamma_\G\Big(s-t,(x,0)^{-1}\star(y,\Psi_{x,y}^{-1}(u))\Big)
  \leq \beta\,N(u)^{-Q}.
\end{equation}
  On the other hand, since $\gamma_\G$ identically vanishes
  on $\{t\leq 0\}$, the above estimate holds
  for every $(t,x)\in \R^{1+n}$ and every
  $(s,y,u)\in\R^{1+n+p}$.
 \end{remark}
 By gathering together Lemma \ref{lem.HtGsaturable},
 Theorem \ref{thm.Gammaint} and Theorem \ref{exTheoremA},
 we are in a position to prove the existence
 of a global fundamental solution for $\Ht$.
 \begin{theorem} [\textbf{Existence of a fundamental solution for $\Ht$}]
   \label{thm.existGammaHt}
 Let $\gamma_\G$, $\Gamma_\G$ and $d$ be as in Theorem \ref{exTheoremC}.
 Then the following function
\begin{equation}\label{defiGammaesplicitina}
  \Gamma(t,x;s,y) := \int_{\R^p}\Gamma_\G(t,x,0;s,y,\eta)\,\d\eta=
    \int_{\R^p}\gamma_\G\Big(s-t,(x,0)^{-1}\star(y,\eta)\Big)\,\d\eta
\end{equation}
   is a fundamental solution for $\Ht$. Moreover, one has the estimates
   \begin{equation*}
    \begin{split}
    \mathbf{c}^{-1}\,(& s-t)^{-Q/2}\,\int_{\R^p}\exp\Bigg(
     -\frac{\mathbf{c}\,d^2\big((x,0)^{-1}\star(y,\eta)\big)}{s-t}\Bigg)\,\d\eta
     \,\,\leq\,\,\Gamma(t,x;s,y) \\
     & \leq \mathbf{c}\,(s-t)^{-Q/2}\,\int_{\R^p}\exp\Bigg(
     -\frac{d^2\big((x,0)^{-1}\star(y,\eta)\big)}{\mathbf{c}\,(s-t)}\Bigg)\,\d\eta,
   \end{split}
   \end{equation*}
   holding true for every $(t,x),(s,y)\in\R^{1+n}$ with $s > t$.
   Here,
   $\mathbf{c} > 0$ is a constant only depending on
   the homogeneous Carnot group $\G$ and on the operator $\mathcal{H}$.
   Finally, $d$ can be replaced by any homogeneous norm on the
   homogeneous Carnot group $\G=(\RN,\star)$.
 \end{theorem}

 \section{Representation formulas for the derivatives}\label{sec.DERIVPbHT}
 In this section, in order to prove Theorem \ref{teomainderiv}, we use
 a quite versatile technique, only based on homogeneity arguments.
 Some of our previous arguments (of dominated-convergence type)
 may be attacked with this technique also; how\-e\-ver, in the previous sections,  we preferred to contain
 the use of homogeneity, in view of future investigations where the latter is not available.

  The key ingredients for the proof of Theorem \ref{teomainderiv}
  are the following technical 
 Lemmas \ref{lemma.derivaaaa} and \ref{lem.liftingunderintegral}
  (where we use the notations in \eqref{eq.dsplitdeladelastar}
  and \eqref{delaprodottosaQQQQQQQQQQ}):
\begin{lemma}\label{lemma.derivaaaa}
 Let $\Omega:=\{(z,\zeta,\eta)\in \R^{1+n}\times \R^{1+n}\times \R^p\,:\,(z,0)\neq (\zeta,\eta)\}$.
 Suppose $g\in C^\infty(\Omega)$ is homogeneous of degree $\alpha<-q^*$ with respect to the family of dilations
 (with our usual notation)
 $$E_\lambda(z,\zeta,\eta)=E_\lambda\Big((t,x),(s,y),\eta\Big)= \Big(\lambda^2 t, \dela(x), \lambda^2 s, \dela(y), \delta^*_\lambda(\eta)\Big). $$
 Let $Z$ be any smooth vector field in the $(z,\zeta)$-variables, homogeneous of positive degree
 with respect to the family of dilations
 $$(z,\zeta)=\big((t,x),(s,y)\big)\mapsto \Big(\lambda^2 t, \dela(x),\lambda^2 s, \dela(y)\Big).$$
 Then, the following facts hold:
 \begin{enumerate}
   \item for any fixed $(z,\zeta)\in \R^{1+n}\times \R^{1+n}$ with $z\neq \zeta$, the map $\eta\mapsto g(z,\zeta,\eta)$
   belongs to $L^1(\R^p)$;
   \item $Z$ can pass under the integral sign as follows
 \begin{equation} \label{eq.ZpassIntegral}
  Z\bigg\{(z,\zeta)\mapsto \int_{\R^p} g(z,\zeta,\eta)\,\d\eta\bigg\}=
  \int_{\R^p} Z\Big\{(z,\zeta)\mapsto g(z,\zeta,\eta)\Big\}\,\d\eta,
  \end{equation}
  for every $z,\zeta\in\R^{1+n}$ with $z\neq\zeta$.
 \end{enumerate}
\end{lemma}
 \begin{proof}
 (1)\,\,Let us fix $z_0,\zeta_0\in\R^{1+n}$ such that $z_0\neq \zeta_0$ and let
 $S,\,N$ be the homogeneous norms introduced in \eqref{sec.two_2:eq_defcanonicalh}. Since, obviously,
 $\eta\mapsto g(z_0,\zeta_0,\eta)$ belongs to $L^1_{\loc}(\R^p)$, we need to prove that
 $$\int_{\{N > 1\}}g(z_0,\zeta_0,\eta)\,\d\eta < \infty.$$
 To this end, we first choose $\rho > 0$ in such a way that $z_0,\zeta_0\in\{S(z)\leq \rho\}$ and we observe that,
 since the set $K:= \{S\leq \rho\}^2\times\{N = 1\}$ is compact and contained in $\Omega$,
 there exists $c > 0$ such that
 \begin{equation} \label{eq.boundgcompact}
  |g(z,\zeta,\eta)|\leq c \qquad\text{for every $z,\zeta\in\{S\leq\rho\}$ and $\eta\in\{N = 1\}$}.
  \end{equation}
 On the other hand, if $\eta\in\R^p$ is such that $N(\eta) > 1$ and if we set $\lambda := 1/N(\eta)\in (0,1)$,
 it is readily seen that $(z_0',\zeta_0',\eta') = E_\lambda(z_0,\zeta_0,\eta)\in K$;
 thus, by \eqref{eq.boundgcompact} and the $E_\lambda$-homogeneity of $g$, we get
 $$|g(z_0,\zeta_0,\eta)| \leq c\,N(\eta)^{\alpha} \quad \text{for every $\eta\in\R^p$ with $N(\eta) > 1$}.$$
 Since $\alpha < -q^*$, the map $\eta\mapsto g(z_0,\zeta_0,\eta)$
 is integrable on $\{N > 1\}$, as desired. \medskip

 (2)\,\,We first prove that, if $Z$ is a smooth vector field as in the statement of the lemma,
 fixing $z,\zeta\in \R^{1+n}$ with $z\neq\zeta$,
 the function
 $$\Phi(\eta):= Z\{(z,\zeta)\mapsto g(z,\zeta,\eta)\}$$
 is $\eta$-integrable on the whole of $\R^p$.

 To this end we observe that, if we think of $Z$ as a vector field
 defined on $\R^{1+n}_z\times \R^{1+n}_\zeta\times \R^p_\eta$ but acting only in the $(z,\zeta)$ variables
 (and not on $\eta$), then
 $Z$ is $E_\lambda$-homogeneous of degree $m$; as a consequence,
 $\Phi$ is $E_\lambda$-homogeneous of degree $\alpha - m$.
 Since, by assumption, $m \geq 0$ and $\alpha < -q^*$, we derive from statement (1) that
 $\Phi(\eta)$  belongs to $L^1(\R^p)$ for every $z,\zeta\in \R^{1+n}$ with $z\neq\zeta$.
   We now turn to prove identity \eqref{eq.ZpassIntegral}. To this aim, we first write
   $$\int_{\R^p}\Phi(\eta)\,\d\eta =
   \int_{\{N(\eta)\leq 1\}}\Phi(\eta)\,\d\eta + \int_{\{N(\eta) > 1\}}\Phi(\eta)\,\d\eta.$$
   We then fix $z_0,\zeta_0\in\R^{1+n}$ such that $z_0\neq\zeta_0$ and we show that
   the function
   $\Phi$ can be dominated, both on $A = \{N\leq 1\}$
   and on $B = \{N > 1\}$, by an integrable function which does not depend of $(z,\zeta)$
   (at least for every $(z,\zeta)$ in a small neighborhood of $(z_0,\zeta_0)$). As for the first set,
   we choose $r > 0$ in such a way that
   $\overline{B(z_0,r)}\cap\overline{B(\zeta_0,r)}=\varnothing$ and we set $K :=
    \overline{B(z_0,r)}\times\overline{B(\zeta_0,r)}\times\{N\leq 1\}$. By the choice of $r$,
    we see that $K$ is a compact subset of $\Omega$; thus,
    there exists a constant $c > 0$ such that
    $$|\Phi(\eta)| =\Big|Z\{(z,\zeta)\mapsto g(z,\zeta,\eta)\}\Big|\leq c,$$
    for every $z,\zeta\in \overline{B(z_0,r)}\times\overline{B(\zeta_0,r)}$
    and every $\eta\in\{N\leq 1\}$.

    As for the set $B$, we argue as in the proof of the previous statement (1):
    if $\rho > 0$ is such that $z_0,\zeta_0\in\{P(z)\leq \rho\}$, from the $E_\lambda$-homogeneity
    of $\Phi$ we infer the existence of another constant
    $c' > 0$ such that
    $$|\Phi(\eta)| =\Big|Z\{(z,\zeta)\mapsto g(z,\zeta,\eta)\}\Big|\leq c'\,N(\eta)^{\alpha-m},$$
    for every $z,\zeta\in\{P\leq 1\}$ and every $\eta\in \{N > 1\}$;
    since
    $\alpha-m \leq \alpha < -q^*$, the function $N^{\alpha-m}$ is integrable
    on $B$. This ends the proof.
   \end{proof}
    \begin{lemma} \label{lem.liftingunderintegral}
    Let $\rho\in C^\infty(\R^{1+N}\setminus\{0\})$ be homogeneous
    of degree $d < -q^*$ with respect to the family
    of dilations \emph{(}see \eqref{eq.dsplitdeladelastar}\emph{)}
    $$F_\lambda(t,x,\xi) := \Big(\lambda^2t, D_\lambda(x,\xi)\Big) =
    \Big(\lambda^2t, \dela(x),\delta_\lambda^*(\xi)\Big).$$
    Then, for every $j = 1,\ldots,m$ we have
    \begin{equation} \label{eq.liftingunderint}
    \begin{split}
     & \int_{\R^p}X_j^y\Big\{y\mapsto \rho\Big(s-t,(x,0)^{-1}\star(y,\eta)\Big)\Big\}\,\d \eta \\
     & \qquad\qquad\qquad = \int_{\R^p}(Z_j\rho)\big(s-t,(x,0)^{-1}\star(y,\eta)\big)\,\d \eta,
     \end{split}
    \end{equation}
    where $Z_j$ is the lifting vector field of $X_j$ as in Theorem \ref{exTheoremB}.
   \end{lemma}
   \begin{proof}
    First of all, by Lemma \ref{lemma.derivaaaa}-(1),
    the two integrand functions
    appearing in \eqref{eq.liftingunderint} are $\eta$-integrable on
    $\R^p$; moreover, since $Z_j$ is a lifting of $X_j$, one has
    $$Z_j^{(y,\eta)} = X_j^y + R_j, \qquad\text{where
    $R_j = \sum_{k = 1}^p r_{j,k}(y,\eta)\,\frac{\de }{\de \eta_k}$},$$
    where $r_{j,k}$ is smooth and $D_\lambda$-homogeneous
    of degree $\sigma^*_k-1$ (see \eqref{eq.explicitdelastar}). In particular,
    $r_{j,k}$ does not depend on $\eta_k$.
    Now, since $Z_j$ is left-invariant on the group $\G = (\RN,\star )$, it is not
    difficult to recognize that\footnote{In fact,
    $Z_j$ is left-invariant on the product group
    $(\R^{1+N},\bullet)$, where
    $$(t,x,\xi)\bullet(s,y,\eta) := (t+s, (x,\xi)\star (y,\eta)).$$}
    $$Z_j^{(y,\eta)}\Big\{(y,\eta)\mapsto
    \rho\Big(s-t,(x,0)^{-1}\star (y,\eta)\Big)\Big\}
    = (Z_j\rho)\Big(s-t,(x,0)^{-1}\star (y,\eta)\Big).$$
    As a consequence, we have the following chain of identities
    \begin{align*}
     & \int_{\R^p}X_j^y\Big\{y\mapsto \rho\big(s-t,(x,0)^{-1}\star (y,\eta)\big)\Big\}\,\d \eta
      \\
      & \quad = \int_{\R^p}(Z_j^{(y,\eta)}-R_j)
      \Big\{(y,\eta)\mapsto \rho\big(s-t,(x,0)^{-1}\star (y,\eta)\big)\Big\}\,\d \eta \\
     & \quad = \int_{\R^p}(Z_j\rho)\big(s-t,(x,0)^{-1}\star (y,\eta)\big)\,\d \eta
     \\
     & \quad\quad\quad-\int_{\R^p}R_j
     \Big\{\eta\mapsto
    \rho\big(s-t,(x,0)^{-1}\star (y,\eta)\big)\Big\}\,\d\eta.
    \end{align*}
    In view of this computation, the desired
    \eqref{eq.liftingunderint} follows if we show that
    \begin{equation} \label{eq.toproveintRnull}
     \int_{\R^p}R_j
     \Big\{\eta\mapsto
    \rho\big(s-t,(x,0)^{-1}\star (y,\eta)\big)\Big\}\,\d\eta = 0.
    \end{equation}
    In its turn, identity \eqref{eq.toproveintRnull}
    can be proved as follows: first of all,
    since $r_{j,k}$ is independent of
    $\eta_k$, by Fubini's theorem we can write
    \begin{align*}
     & \int_{\R^p}R_j
     \Big\{\eta\mapsto
    \rho\big(s-t,(x,0)^{-1}\star (y,\eta)\big)\Big\}\,\d\eta
    \\
    & \quad = \sum_{k = 1}^p
    \int_{\R^p}r_{j,k}(y,\eta)\,\frac{\de }{\de \eta_k}
     \Big\{\eta\mapsto
    \rho\big(s-t,(x,0)^{-1}\star (y,\eta)\big)\Big\}\,\d\eta \\
    & \quad = \sum_{k = 1}^p
    \int_{\R^p}\frac{\de }{\de \eta_k}
     \Big\{r_{j,k}(y,\eta)\,
    \rho\big(s-t,(x,0)^{-1}\star (y,\eta)\big)\Big\}\,\d\eta \\
    & \quad = \sum_{k = 1}^p
    \int_{\R^{p-1}}\bigg(\int_{-\infty}^\infty\frac{\de }{\de \eta_k}
     \Big\{r_{j,k}(y,\eta)\,
    \rho\big(s-t,(x,0)^{-1}\star (y,\eta)\big)\Big\}\,\d\eta_k\bigg)\d\widehat{\eta}_k,
    \end{align*}
    where $\widehat\eta_k$ denotes the $(p-1)$-tuple of variables obtained by removing
    $\eta_k$ from $\eta$.
    On the other hand, since $\rho$ vanishes at infinity
    (as it is $F_\lambda$-homogeneous of
    negative degree) and since
    $\|(x,0)^{-1}\star (y,\eta)\|\to\infty$ as $\eta_k\to\pm\infty$, one has
    \begin{align*}
     & \lim_{\eta_k\to\pm\infty}
     r_{j,k}(y,\eta)\,
    \rho\big(s-t,(x,0)^{-1}\star (y,\eta)\big)\\
    & \qquad
    =r_{j,k}(y,\eta)\cdot\lim_{\eta_k\to\pm\infty}
    \rho\big(s-t,(x,0)^{-1}\star (y,\eta)\big) = 0.
    \end{align*}
    This ends the proof.
   \end{proof}

   Thanks to Lemmas \ref{lemma.derivaaaa} and \ref{lem.liftingunderintegral},
   we can now provide the
   \begin{proof}[of Theorem \ref{teomainderiv}]
    For the sake of readability, we split the proof of
    formulas \eqref{eq.derGammays}-to-\eqref{eq.derGammatutte}
    into three different steps. \medskip

    \textsc{Step I:} We first prove formula
    \eqref{eq.derGammays}. To this end
    we observe that, by repeatedly applying Lemma \ref{lemma.derivaaaa},
    we have the representation
    \begin{align} 
     & \Big(\frac{\de}{\de s}\Big)^\alpha\Big(\frac{\de}{\de t}\Big)^\beta 
     X_{i_1}^y\cdots X_{i_h}^y
     \Gamma(t,x;s,y)\label{eq.derGammaysintermediate}\\
  & \quad = \int_{\R^p}\Big(\frac{\de}{\de s}\Big)^\alpha
     \Big(\frac{\de}{\de t}\Big)^\beta X_{i_1}^y\cdots X_{i_h}^y
  \Big\{(t,s,y)\mapsto
  \gamma_\G\big(s-t,(x,0)^{-1}\star (y,\eta)\big)\Big\}\,\d\eta \notag\\
  & \quad =
  (-1)^{\beta}
  \int_{\R^p}X_{i_1}^y\cdots X_{i_h}^y 
  \Big\{y\mapsto\big((\de_s)^\alpha
  (\de_t)^\beta \gamma_\G\big)\big(s-t,(x,0)^{-1}\star (y,\eta)\big)\Big\}\,\d\eta.\notag
    \end{align}
   Formula \eqref{eq.derGammays} can now be obtained from
    \eqref{eq.derGammaysintermediate} by repeatedly applying
    Lemma \ref{lem.liftingunderintegral}: in fact,
    on account of Theorem \ref{exTheoremC}-(iii)
    we know that the functions
    \begin{align*}
    & h_1 = (\de_s)^\alpha
  (\de_t)^\beta \gamma_\G,\quad h_2 = Z_{i_h}(\de_s)^\alpha
  (\de_t)^\beta\gamma_\G,\quad\ldots,\quad \\
  & \qquad\quad h_{h+1} = Z_{i_2}\cdots Z_{i_h}
  (\de_s)^\alpha
  (\de_t)^\beta\gamma_\G
  \end{align*}
  are smooth on $\R^{1+N}\setminus\{0\}$ and $F_\lambda$-homogeneous of degrees
  \begin{align*}
   & d_1 = -Q-2 \alpha -2 \beta ,\quad d_2 = -Q-2 \alpha -2 \beta -1,\quad \ldots,\\
   & \qquad\qquad d_{h+1} = -Q-2 \alpha -2 \beta -h+1,
   \end{align*}
  respectively. Since $Q = q+q^*$,
  we clearly have $d_1,\ldots,d_{h+1} < -q^*$. \medskip

  \textsc{Step II:} We prove formula \eqref{eq.derGammatx}.
  To this end, in order to apply Lemma \ref{lem.liftingunderintegral},
  we first introduce the following map:
  $$\phi_{x,y}:\R^p\to\R^p, \qquad\phi_{x,y}(u) := \pi_p\Big((x,0)\star (x,u)^{-1}\star (y,0)\Big),$$
  where $\pi_p$ is the projection
  of $\RN=\R^n\times\R^p$ onto $\R^p$. By exploiting
  the $D_\lambda$-homogeneity of the component functions of $\star $,
  it is not difficult to check that $\phi_{x,y}$ is a smooth diffeomorphism
  of $\R^p$, further satisfying
  $$\mathrm{det}\big|\mathcal{J}_{\phi_{x,y}}(u)\big| = 1, \quad\text{for every
  $u\in\R^p$}.$$
  Moreover, by using the explicit construction of the group $\G$
  in Theorem \ref{exTheoremA} (see \cite{BiagiBonfiglioliLast}
  for all the details), one can prove
  that
  $$(x,0)^{-1}\star (y,\phi_{x,y}(u)) = (x,u)^{-1}\star (y,0),\quad \forall\,\,x,y\in\R^n,\,u\in\R^p.$$
  Gathering together the above facts,
  and performing the change of variable $\eta = \phi_{x,y}(u)$, we then obtain
  the following alternative representation of $\Gamma$
  (also remind of the symmetry of $\gamma_\G$, see Theorem \ref{exTheoremC}-(ii)):
  \begin{equation} \label{eq.startingnewGamma}
  \begin{split}
    \Gamma(t,x;s,y) &= \int_{\R^p}\gamma_\G\big(s-t,(x,u)^{-1}\star (y,0)\big)\,\d u \\
   &= \int_{\R^p}\gamma_\G\big(s-t,(y,0)^{-1}\star (x,u)\big)\,\d u.
   \end{split}
   \end{equation}
   Now, starting from \eqref{eq.startingnewGamma} and
   repeatedly using Lemma \ref{lemma.derivaaaa}, we get
   \begin{equation*}
    \begin{split}
     & \Big(\frac{\de}{\de s}\Big)^\alpha\Big(\frac{\de}{\de t}\Big)^\beta 
     X_{j_1}^x\cdots X_{j_k}^x
     \Gamma(t,x;s,y)\\
  & \quad = \int_{\R^p}
  \Big(\frac{\de}{\de s}\Big)^\alpha\Big(\frac{\de}{\de t}\Big)^\beta 
   X_{j_1}^x\cdots X_{j_k}^x
  \Big\{(t,s,x)\mapsto
  \gamma_\G\big(s-t,(y,0)^{-1}\star (x,u)\big)\Big\}\,\d u \\
  & \quad =
  (-1)^{\beta}
  \int_{\R^p}X_{j_1}^x\cdots X_{j_k}^x \Big\{x\mapsto\Big((\de_s)^\alpha
  (\de_t)^\beta \gamma_\G\Big)\big(s-t,(y,0)^{-1}\star (x,u)\big)\Big\}\,\d u.
    \end{split}
   \end{equation*}
   From this, by repeatedly applying Lemma \ref{lem.liftingunderintegral} (with $x$ in place of
   $y$) and by arguing exactly as in the previous step,
   we obtain the desired \eqref{eq.derGammatx}. \medskip

   \textsc{Step III:} We finally prove formula \eqref{eq.derGammatutte}.
   To begin with, we use \eqref{eq.derGammays} and the change
   of variable $\eta = \phi_{x,y}(u)$ introduced in Step II to write
   \begin{align*}
    & \Big(\frac{\de}{\de s}\Big)^\alpha\,\Big(\frac{\de}{\de t}\Big)^\beta
    \,X^y_{i_1}\cdots X^y_{i_h}
   \Gamma(t,x;s,y) \\
   & \quad = (-1)^{\beta}\int_{\R^p}
   \bigg({Z}_{i_1}\cdots {Z}_{i_s}
   \Big(\frac{\de}{\de s}\Big)^\alpha\,\Big(\frac{\de}{\de t}\Big)^\beta\gamma_\G \bigg)\Big(s-t,
   (x,u)^{-1}\star (y,0)\Big) \,\d u \\
   & \quad \big(\text{setting $\rho := \big({Z}_{i_1}
   \cdots {Z}_{i_h}(\de_s)^\alpha(\de_t)^\beta\gamma_\G \big)\circ\widetilde{\iota}$}\,\big) \\
   & \quad = (-1)^{\beta}\int_{\R^p}\rho
   \big(s-t,
   (y,0)^{-1}\star (x,u)\big) \,\d u.
   \end{align*}
    From this, by repeatedly applying Lemma \ref{lem.liftingunderintegral}
    and by arguing exactly in Step II (notice that
    $\rho,Z_{i_k}\rho,\ldots,Z_{j_2}\cdots Z_{j_k}\rho$
    are all $F_\lambda$-homogeneous of degree less than $-q^*$), we
    obtain the desired \eqref{eq.derGammatutte}. This ends the proof.
   \end{proof}
 \section{An application to the Cauchy problem for $\Ht$}
 \label{sec.CauchyPbHT}
 In this section we turn our attention to the Cauchy problem for $\Ht$.
 In doing this, we shall use many of the properties of $\Gamma$
 in Theorem \ref{mainteo}, whose proof is postponed to Section \ref{sec.propertGamma}.

 To begin with, let $\varphi\in C(\R^n)$ and $\Omega =(0,\infty)\times \R^n$.
   We say that a function $u:\Omega\to\R$ is a (classical)
   solution of the Cauchy problem
   \begin{equation} \label{eq.CauchyPbHt}
    \begin{cases}
     \Ht u = 0 & \text{in $\Omega$} \\
     u(0,x) = \varphi(x) & \text{for $x\in\R^n$,}
    \end{cases}
   \end{equation}
   if the following conditions are satisfied:
  $u\in C^2(\Omega)$ and $\Ht u= 0$
    on $\Omega$; $u$ is continuous up to $\overline{\Omega}$ and $u(0,\cdot) = \varphi$
    pointwise on $\R^n$. By the $C^\infty$-hypoellipticity of $\Ht$, any classical solution
  of \eqref{eq.CauchyPbHt} is smooth on $\Omega$.
  The following theorem is the main result of this section.
  \begin{theorem} \label{thm.solCauchyPb}
   In the above notations,  if $\varphi$ is continuous and bounded, then
   \begin{equation} \label{eq.defuintegral}
    u: \Omega \longto\R
    \qquad u(t,x) := \int_{\R^n}\Gamma(0,y;t,x)\,\varphi(y)\,\d y
   \end{equation}
   is the unique bounded
   classical solution of \eqref{eq.CauchyPbHt};
   furthermore, it satisfies
     \begin{equation} \label{eq.uboundedsolCauchy}
    \sup_{\Omega} |u| \leq \sup_{\R^n} |\varphi|.
   \end{equation}
  \end{theorem}
  \begin{proof}
 Since the uniqueness problem is of independent interest
 (and since we prove it with a totally different technique),
 this is postponed to Proposition \ref{prop.uniquesol}. Then we focus on the
 rest of the assertion.\medskip

 \noindent  First of all, by (ii), (vii) in Theorem \ref{mainteo},
 $u$ is well posed and it satisfies \eqref{eq.uboundedsolCauchy}:
   indeed, for $t>0$,
   \begin{align*}
   |u(t,x)|\leq  \|\varphi\|_{\infty}\,\int_{\R^n}\Gamma(0,y;t,x)\,\d y= \|\varphi\|_{\infty}\,\int_{\R^n}\Gamma(0,x;t,y)\,\d y =
    \|\varphi\|_{\infty}.
   \end{align*}
   The rest of the proof is split in three steps. \medskip

   \textsc{Step I:}\,\,In this step we prove that
   $u\in C(\Omega)$.
   To this end, let $z_0 = (t_0,x_0)$
   be a fixed point in $\Omega$ and let $r > 0$ be such that
   $$K:= [t_0-r,t_0+r]\times\overline{B}(x_0,r)\subseteq\Omega.$$
   Moreover, let $z_n\to z_0$; we can assume that $z_n\in K$. By arguing
   as in the proof of Lemma \ref{lem.intgralGammafur}-(ii),
   one can easily recognize that
   $$(y,\eta)\mapsto \Gamma_\G(0,y,0;t,x,\eta)\quad \text{is in $L^1(\R^{N})$, for every $(t,x)\in\R^{1+n}$}.$$
   Therefore, by Fubini's theorem, for every $n\geq 0$ we
   can write
   \begin{align*}
    u(z_n) &= u(t_n,x_n) \stackrel{\eqref{eq.defiGammaGHtG}}{=}
    \int_{\R^n\times \R^p}\gamma_\G\big(t_n,(y,0)^{-1}\star(x_n,\eta)\big)\,
    \varphi(y)\,\d y\,\d\eta\\
    &= \int_{\R^n\times \R^p}\gamma_\G(t_n,u,v)\,\varphi\big(C_{x_n}^{-1}(u,v)\big)
    \,\d u\,\d v,
      \end{align*}
  where we have use the smooth diffeomorphism $C_x(y,\eta) := (y,0)^{-1}\star(x,\eta)$
  (whose Jacobian de\-ter\-mi\-nant 
  is $1$). A dominated convergence argument is now in order; we skip the details,
  apart from the non-trivial estimate (based on the Gaussian bound in \eqref{eq.Gaussest})
   \begin{align*}
    & \big|\gamma_\G(t_n,u,v)\,\varphi\big(C_{x_n}^{-1}(u,v)\big)\big| \\
    & \quad \leq \mathbf{c}\,(t_0-r)^{-Q/2}\,\|\varphi\|_{\infty}\,
    \exp\left(-\frac{d^2(u,v)}{\mathbf{c}\,(t_0-r)}\right)=:f(u,v).
   \end{align*}
   In turn, the integrability of $f$ is ensured by the estimate
   $$f(u,v) \leq
   \exp\left(-\frac{\alpha\,S^2(u)}{\mathbf{c}\,(t_0-r)}\right)
   \exp\left(-\frac{\alpha\,N^2(v)}{\mathbf{c}\,(t_0-r)}\right),
   $$
   where $S,N$ are as in the proof of Theorem \ref{thm.Gammaint}
   (by arguing as in the few lines after \eqref{eq.toprovestatibis}, one
   gets the integrability of the above right-hand side).\medskip

   \textsc{Step II:}\,\,Since $u$ is continuous by Step I,
   if we show that $\Ht u = 0$ in $\mathcal{D}'(\Omega)$,
   the hypoellipticity of $\Ht$ will imply that   $u\in C^\infty(\Omega)$ and $\Ht u= 0$
   on $\Omega$. To this end, let $\psi\in C_0^\infty(\Omega)$.
   We have
   \begin{align*}
    & \int_{\R^{1+n}}u(t,x)\,\Ht^*\psi(t,x)\,\d t\,\d x \\
    & \qquad \,\,\,=\,\,\,\int_{\R^n}\left(\int_{\R^{1+n}}\Gamma(0,y;t,x)\,\Ht^*\psi(t,x)
    \,\d t\,\d x  \right)\varphi(y)\,\d y \\
    & \qquad \stackrel{\eqref{sec.two:eq_GreenGamma}}{=}
     -\int_{\R^n}\psi(0,y)\,\varphi(y)\,\d y =0
   \end{align*}
   Here we applied  Fubini's Theorem, whose legitimacy is due to the estimate (see also
   (vii) in Theorem \ref{mainteo})
   \begin{align*}
    \int_{\textrm{supp}(\psi)}\left(\int_{\R^n}\Gamma(0,x;t,y)\,\d y\right)\,\d t\,\d x
     \leq \meas(\textrm{supp}(\psi)) < \infty.
   \end{align*}

   \textsc{Step III:}\,\,To end the proof, we must show that $u$ satisfies the needed initial condition. To this end, let $x\in\R^n$
   be fixed and let $t_n  \in (0,1)$ be vanishing, as $n\to \infty$.
  Arguing as in Step I (and with the aid of
  (ii) and (vii) of Theorem \ref{mainteo}), one gets
   \begin{align*}
    & |u(t_n,x)-\varphi(x)|  \leq
    \int_{\R^n\times \R^p}\gamma_\G(t_n,u,v)\,
    \big|\varphi\big(C_x^{-1}(u,v)\big)-\varphi(x)\big|\,
    \d u\,\d v \\
    & \quad \stackrel{\eqref{eq.Gaussest}}{\leq}
    \mathbf{c}\,(t_n)^{-Q/2}\,
    \int_{\R^n\times \R^p}\exp\left(-\frac{d^2(u,v)}{\mathbf{c}\,t_n}\right)
    \big|\varphi\big(C_x^{-1}(u,v)\big)-\varphi(x)\big|\,
    \d u\,\d v\\
    & \quad \,\,\,=\,\,\mathbf{c}\,\int_{\R^n\times \R^p}\exp\left(-\frac{d^2(u',v')}{\mathbf{c}}\right)\,
   \big|\varphi\big((C_x^{-1}\circ D_{\sqrt{t_n}})(u',v')\big)-\varphi(x)\big|\,
    \d u'\,\d v'.
   \end{align*}
  In the last equality we used
  the change of variable
   $(u,v) = D_{\sqrt{t_n}}(u',v')$, and
   $D_\lambda$-homogeneity of $d$.
   Since one clearly has (due to the continuity of $\varphi$)
    $$\lim_{n\to\infty}\varphi\big((C_x^{-1}\circ D_{\sqrt{t_n}})(w,z)\big)
    = \varphi\big(C_x^{-1}(0,0)\big) = \varphi(x),$$
    we deduce that
    $u(t_n,x)\to \varphi(x)$, thanks to a dominated-convergence argument (see Step I) based on
    \begin{equation*}
     \begin{split}
      & \exp\left(-\frac{d^2(u',v')}{\mathbf{c}}\right)\,
   \big|\varphi\big((C_x^{-1}\circ D_{\sqrt{t_n}})(u',v')\big)
   -\varphi(x)\big| \\
   & \qquad \leq 2\,\|\varphi\|_{\infty}\,\exp\left(-\frac{d^2(u',v')}{\mathbf{c}}\right).
     \end{split}
    \end{equation*}
  This ends the proof.
  \end{proof}
   We now turn to the \emph{uniqueness} of the solution of the Cauchy problem
   for $\Ht$:
   \begin{proposition}\label{prop.uniquesol}
     The only bounded classical solution of \eqref{eq.CauchyPbHt} when $\varphi\equiv 0$
     is the null function.
     As a consequence, \eqref{eq.defuintegral} is the unique bounded solution of \eqref{eq.CauchyPbHt}.
   \end{proposition}
   \begin{proof}
    Let $u$ be a bounded classical solution
    of the homogeneous Cauchy pro\-blem for $\Ht$,
    and let
    $v(t,x,\xi):= u(t, x)$ defined on $\R\times \R^n\times \R^p$.

   Clearly, $v(0,x,\xi) = u(0,x) = 0$ for every $(x,\xi)\in\R^n\times \R^p$;
   moreover, since $\Ht_\G$ is a lifting of $\Ht$
    on $\R\times\R^n\times \R^p$, we get
 $\Ht_\G v = \Ht u = 0$ point-wise
    on $(0,\infty)\times\R^n\times \R^p$.
  Summing up, $v$ is a bounded solution
   of the homogeneous Cauchy problem for $\Ht_\G$.
   Since we have transferred our setting to that of Carnot groups $\G$,
   we are consequently entitled to apply
   \cite[Theorem 2.1]{BLUpaper}, which ensures that
   $v\equiv 0$, and this ends the proof.
   \end{proof}

\section{Further properties of $\Gamma$}\label{sec.propertGamma}
  This appendix is completely devoted to establishing the properties
  (i)-to-(x) of $\Gamma$ in Theorem \ref{mainteo}. Throughout the section, $\Gamma$ is as in
  \eqref{defiGammaesplicitinaPRE1} and all the notations used so far are tacitly understood.\medskip

  \noindent Some of the properties we aim to prove are consequences of Theorem \ref{exTheoremC}:
\begin{itemize}
 \item
 (i) is a trivial consequence of the integral form of $\Gamma$
 in \eqref{defiGammaesplicitinaPRE2} jointly with (i) in Theo\-rem~C.

\item
 The first part of (ii)
 comes from \eqref{defiGammaesplicitinaPRE2}; the symmetry in $x,y$ will be proved later on.

 \item
 (iii) follows from (iii) of Theorem \ref{exTheoremC} together with the change
  of variable $\eta= \delta^*_\lambda(\eta')$
  (see also  \eqref{eq.dsplitdeladelastar} and \eqref{delaprodottosaQQQQQQQQQQ}).

 \item
 (vii) follows from (v) of Theorem \ref{exTheoremC} by making use of
 the change
  of variable $(y,\eta) = (x,0)\star(y',\eta')$;

  \item
  (ix) has been proved in Section \ref{sec.CauchyPbHT}.
\end{itemize}
 Despite the simplicity of its statement, the proof of the following proposition is technical
 and is a prototype for many of the next proofs.
\begin{proposition} \label{prop.regGammaHt}
 The following facts hold true:
  \begin{itemize}
   \item[\emph{(a)}] $\Gamma$ is continuous out of the diagonal
   of $\R^{1+n}\times\R^{1+n}$.

   \item[\emph{(b)}] For every fixed compact set
   $K\subseteq\R^{1+n}$, we have
  $$\text{ $\sup_{z\in K}\Gamma(z;\zeta)\to 0$
   as $\|\zeta\|\to\infty$.}$$
   \item[\emph{(c)}] For every fixed $\zeta \in\R^{1+n}$, we have
   $
    \Gamma(z;\zeta)\to 0$  as $\|z\|\to\infty$.
      \end{itemize}
 \end{proposition}
 \begin{proof}
  	(a)\,\,
   It is a dominated-convergence argument applied to the limit
   $$\lim_{n\to\infty}\Gamma(z_n;\zeta_n)  =\lim_{n\to\infty}
   \int_{\R^p}
  	 \gamma_\G(s_n-t_n,(x_n,0)^{-1}\star(y_n,\eta))\,\d\eta,$$
  where $z_n=(t_n,x_n)\to z_0$, $\zeta_n=(s_n,y_n)\to \zeta_0$ and $z_0\neq \zeta_0$; this argument
  is based on the ingredients:
   \begin{itemize}
     \item[-] a proper use of the change of variable $\eta=\Psi^{-1}_{x_n,y_n}(\eta')$  in
     Remark \ref{sec.two_2:remChangeF};
     \item[-] the continuity of $\gamma_\G$ out of the origin of $\R^{1+N}$;
     \item[-] the bound \eqref{eq.crucialestim} in Rem.\,\ref{eq.crucialestim}
     (together with the integrability of $N^{-Q}(\eta')$ on the set
     $\{N(\eta')>1\}$).
   \end{itemize}
	
  	(b)\,\,   It is dominated-convergence, applied to the right-hand limit
   $$\lim_{n\to\infty}\sup_{z\in K}\Gamma(z;\zeta_n)  \leq
   \lim_{n\to\infty}
   \int_{\R^p}
  	\sup_{(t,x)\in K} \gamma_\G(s_n-t,(x,0)^{-1}\star(y_n,\eta))\,\d\eta,$$
  where $z=(t,x)$, $\zeta_n=(s_n,y_n)\to \infty$ and $K$ is compact in $\R^{1+n}$; we also used:
   \begin{itemize}
     \item[-] another use of the change of variable $\eta=\Psi^{-1}_{x,y_n}(\eta')$;

     \item[-] the vanishing of $\gamma_\G$ at infinity (see (iv) in Theorem \ref{exTheoremC}), together with the
     change of variable $\eta = \Phi_{x,y_n}(\eta')$ and the fact that
\begin{equation}\label{infinitosanxz}
   \lim_{n\to\infty}\|(s_n-t, (x,0)^{-1}\star(y_n,\Psi_{x,y_n}(\eta')))\|=\infty,
\end{equation}
   uniformly for $z\in K$ and $\eta'\in\R^p$;
     \item[-] the bound \eqref{eq.crucialestim} in Rem.\,\ref{eq.crucialestim}.
   \end{itemize}

  	(c)\,\, This is similar to (b);
  \eqref{infinitosanxz} is replaced by the (weaker)
  information
  $$   \lim_{n\to\infty}\|(s-t_n, (x_n,0)^{-1}\star(y,\Psi_{x_n,y}(\eta')))\|=\infty\quad \text{uniformly for $\eta'\in\R^p$},$$
  for any fixed $\zeta=(s,y)$. This ends the proof.
 \end{proof}
 \begin{corollary}\label{cor.GammaHtSmooth}
  For every fixed $z\in\R^{1+n}$, the map
  $\zeta\mapsto\Gamma(z;\zeta)$
  is smooth and $\Ht$-harmonic on
  $\R^{1+n}\setminus\{z\}$ (i.e., $\Ht(\Gamma(z;\cdot))=0$ on $\R^{1+n}\setminus\{z\}$).
 \end{corollary}
 \begin{proof}
 By the $C^\infty$-hypoellipticity of $\Ht$, we infer that
 $\Gamma(z;\cdot)$ coincides almost everywhere with a smooth
 $\Ht$-harmonic function on  $\R^{1+n}\setminus\{z\}$;
 the `almost-everywhere' can be dropped in view of (a)
 in Proposition \ref{prop.regGammaHt}.
 \end{proof}
 The following results (a) and (c) establish property (vi) of Theorem \ref{mainteo}, whereas (b)
 is technical for the study of the Cauchy problem for $\Ht$.
 \begin{lemma}\label{lem.intgralGammafur}
The following facts hold true:
  \begin{itemize}
   \item[\emph{(a)}] $\Gamma\in L^1_{\loc}(\R^{1+n}\times\R^{1+n})$.

   \item[\emph{(b)}] For every fixed $(s,y)\in\R^{1+n}$, we have
   \begin{equation} \label{eq.toprovebestat}
   (t,x,\eta)\mapsto \Gamma_\G(t,x,0;s,y,\eta)\in L^1_{\loc}(\R^{1+n+p}).
   \end{equation}

   \item[\emph{(c)}] For every
   fixed $\zeta\in\R^{1+n}$, we have $\Gamma(\cdot;\zeta)\in
   L^1_{\loc}(\R^{1+n})$.
  \end{itemize}
 \end{lemma}
 \begin{proof}
  (a)\,\,Let $K_1,K_2\subseteq\R^{1+n}$ be compact sets
  and let $T > 0$ be so large that $K_2\subseteq[-T,T]\times\R^{n}$.
  By Tonelli's Theorem and (vii) in Theorem \ref{mainteo}, we have
  \begin{align*}
   \int_{K_1\times K_2}\Gamma(z;\zeta)\,\d z\,\d \zeta  \leq 2\,T\,\meas(K_1).
  \end{align*}

  (b)\,\,Let $\zeta = (s,y)\in\R^{1+n}$, and let $K\subseteq\R^{1+N}$ be compact.
  It can be proved (see Remark \ref{sec.two_2:remChangeF}) that the map
  $$H_y:\R^{1+n+p}\longto\R^{1+n+p} \quad H_y(t,x,\eta) :=
  \big(s-t, (x,0)^{-1}\star(y,\eta)\big)$$
  is a smooth diffeomorphism with identically $1$ Jacobian determinant. Therefore
  \begin{align*}
   \int_{K}\Gamma_\G(t,x,0;s,y,\eta)\,\d t\,\d x\,\d\eta
   & = \int_{K}\gamma_\G\big(s-t, (x,0)^{-1}\star(y,\eta)\big)
   \,\d t\,\d x\,\d\eta \\
   & = \int_{H_y^{-1}(K)}\gamma_\G(\tau,z)\,\d\tau\,\d z<\infty,
  \end{align*}
  since $\gamma_\G$ is locally integrable
  and $H_y^{-1}(K)$ is compact.\medskip

  (c)\,\,Let $K\subseteq\R^{1+n}$ be a compact set.
  The map $T(t,x,u) := \big(t,x,\Psi_{x,y}^{-1}(u)\big)$
  is a diffeomorphism of $\R^{1+n+p}$ with Jacobian determinant equal to $1$. Thus
  \begin{align*}
   \int_{K}\Gamma(z;\zeta)\,\d z
   & = \int_{K\times\R^p}\gamma_\G
   \big(s-t,(x,0)^{-1}\star(y,\Psi_{x,y}^{-1}(u))\big)\,\d t
   \,\d x\,\d u
   \\
   & = \int_{K\times\{N\leq 1\}}\{\cdots\}\,\d t
   \,\d x\,\d u
   + \int_{K\times\{N> 1\}}\{\cdots\}\,\d t
   \,\d x\,\d u  =: \mathrm{I} + \mathrm{II},
  \end{align*}
  where $N$ is as in \eqref{sec.two_2:eq_defcanonicalh}.
  \eqref{eq.toprovebestat} implies $\mathrm{I}<\infty$, and
  \eqref{eq.crucialestim} gives $\mathrm{II}<\infty$.
 \end{proof}
  Thanks to Lemma \ref{lem.intgralGammafur},
  we can prove property (viii) of Theorem \ref{mainteo}:
  \begin{proposition}\label{prop.GammainvHtstar}
 For every fixed
  $\varphi\in C^\infty_0(\R^{1+n})$, the function
  \begin{equation*}
   \Lambda_\varphi:\R^{1+n}\longto\R,
   \qquad \Lambda_\varphi(\zeta) := \int_{\R^{1+n}}
   \Gamma(z;\zeta)\,\varphi(z)\,\d z
  \end{equation*}
  is well defined and it satisfies the following properties:
  \begin{itemize}
   \item[\emph{(a)}] $\Lambda_\varphi\in C^\infty(\R^{1+n})$
   and $\Ht(\Lambda_\varphi) = -\varphi$ on $\R^{1+n}$;

   \item[\emph{(b)}] $\Lambda_\varphi(\zeta) \longto 0$ as $\|\zeta\|\to\infty$;

   \item[\emph{(c)}]
  for every $\zeta\in\R^{1+n}$, we have
  \begin{equation*}
   \Lambda_{\Ht\varphi}(\zeta) = \int_{\R^{1+n}}\Gamma(z;\zeta)\,\Ht\varphi(z)
   \,\d z = -\varphi(\zeta).
  \end{equation*}
  \end{itemize}
 \end{proposition}
 \begin{proof}
  By Lemma \ref{lem.intgralGammafur}-(c),
  $\Lambda_\varphi$ is well-defined.   Property (b) is a consequence of Proposition \ref{prop.regGammaHt}-(b).
  By the $C^\infty$-hypoellipticity of $\Ht$,
  (a) will follow if we show that
  $\Lambda_\varphi$ is continuous and
  $\Ht(\Lambda_\varphi) = -\varphi$
  in the sense of distributions.
  To begin with, let $\zeta_n=(s_n,y_n)\to\zeta_0=(s_0,y_0)$.
  Let  $T > 0$ be so large that
  $\mathrm{supp}(\varphi)\subseteq [-T,T]\times\R^n$. We then have
  \begin{align*}
   \Lambda_\varphi(\zeta_n)  &= \int_{s_n-T}^{s_n+T}\int_{\RN}\gamma_\G(\tau,u,v)\,
   \varphi\big(s_n-\tau, C_{y_n}^{-1}(u,v)\big)\,\d\tau\,\d u\,\d v\\
   &=\int_{[-T_0,T_0]\times\RN}
  \gamma_\G(\tau,u,v)\,
   \varphi\big(s_n-\tau, C_{y_n}^{-1}(u,v)\big)\,\d\tau\,\d u\,\d v,
  \end{align*}
  where $C_y$ is as in the proof of Theorem \ref{thm.solCauchyPb}, and
  $T_0 \gg 1$ satisfies
  $$\text{$[s_n-T,s_n+T]\subseteq [-T_0,T_0]$ for any $n$.}$$
  We can now get $\Lambda_\varphi(\zeta_n)\to \Lambda_\varphi(\zeta_0)$
  by a standard dominated convergence argument, based on
  the integrability of $\gamma_\G$ on the strip $[-T_0,T_0]\times\RN$ (see Theorem \ref{exTheoremC}-(v)).
  Finally, $\Ht(\Lambda_\varphi) = -\varphi$
  in $\mathcal{D}'(\R^{1+n})$ is a consequence of the definition
  of fundamental solution (and of Lemma \ref{lem.intgralGammafur}-(a)).\medskip

  We prove (c). We consider $u:= \Lambda_{\Ht\varphi}+\varphi$.
  From property (a), we see that $u$ is smooth and $\Ht$-harmonic on $\R^{1+n}$;
  moreover, from (b) we get that $u$ vanishes at infinity. Since $\Ht$
  satisfies the Weak Maximum Principle on every bounded open set
  (and therefore on space as well; see \cite[Corollary 5.13.7]{BLUlibro}),
  we conclude that $u\equiv 0$ throughout $\R^{1+n}$, as desired.
  \end{proof}
 \begin{theorem}[\textbf{Fundamental Solution for $\Ht^*$}]\label{thm.existGammastar}
 The function $$\Gamma^*(z;\zeta) := \Gamma(\zeta;z)$$
  is a global fundamental solution for the adjoint operator
  $\Ht^* = \LL + \de_t$.
 \end{theorem}
  \begin{proof}
  This follows immediately from (c)
  of Proposition \ref{prop.GammainvHtstar}.
 \end{proof}
  We can now prove property (iv) of Theorem \ref{mainteo}.
  \begin{theorem} \label{thm.GammaHtSmooth}
    $\Gamma$ is smooth out of the diagonal
   of $\R^{1+n}\times\R^{1+n}$.
  \end{theorem}
  \begin{proof}
   We consider the PDO on $\R^{1+n}\times\R^{1+n}$ defined by
   $$Q:= \sum_{j = 1}^m {X}_j^2(x) + \de_t
   + \sum_{j = 1}^m {X}_j^2(y) - \de_s,$$
   where $x,y\in\R^n$ and $t,s\in\R$.
   Obviously, $Q$ is a H\"ormander operator on $\R^{1+n}\times\R^{1+n}$ since this
   is true of $\sum_{j = 1}^m {X}_j^2$ on $\R^n$. By Theorem \ref{thm.GammaHtSmooth}
   we deduce that, for any $(t,x)\neq (s,y)$, one has
\begin{align*}
   Q\big(\Gamma(t,x;s,y)\big) &=
   \Ht^*\Big((t,x)\mapsto \Gamma(t,x;s,y)\Big)+ \Ht\Big((s,y)\mapsto \Gamma(t,x;s,y)\Big)\\
   &=\Ht^*\Big((t,x)\mapsto \Gamma^*(s,y;t,x)\Big)+ \Ht\Big((s,y)\mapsto \Gamma(t,x;s,y)\Big)=0.
\end{align*}
 The $C^\infty$-hypoellipticity of $Q$ and the continuity of $\Gamma$ out of the diagonal
 prove the thesis.
  \end{proof}
  The next result establishes the second part of property (ii) of Theorem \ref{mainteo}.
  \begin{theorem}\label{thm.symmetryGammax}
  For every $(t,x), (s,y)\in\R^{1+n}$ we have
   $$\Gamma(t,x;s,y) = \Gamma(t,y;s,x).$$
  \end{theorem}
  \begin{proof}
   To ease the reading, we split the proof into two steps.
   \medskip

   \textsc{Step I:}\,\,We first prove that the function
   $G$ defined by
   $$ G(t,x;s,y) := \Gamma(t,y;s,x)$$
  is a global fundamental solution for $\Ht$, i.e.,
  for every fixed $z=(t,x)\in\R^{1+n}$,
  \begin{itemize}
   \item[(a)] $G(z;\cdot)\in L^1_{\loc}(\R^{1+n})$;
   \item[(b)] $\Ht G(z;\cdot) = -\mathrm{Dir}_z$ in $\mathcal{D}'(\R^{1+n})$.
  \end{itemize}
  As for assertion (a), let
  $K\subseteq\R^{1+n}$ be a compact set and let $T> 0$ be such that
  $K\subseteq [-T,T]\times\overline{B(0,T)} =: C(T)$.
  Since $\Gamma \geq 0$ and $\Gamma(\cdot;\zeta)\in L^1_{\loc}(\R^{1+n})$
  for every $\zeta\in\R^{1+n}$, one then has
  \begin{align*}
   \int_{K}G(t,x;s,y)\,\d s\,\d y &
   \leq \int_{C(T)}\Gamma(t-s,y;0,x)\,\d s\,\d y \\
   & = \int_{t-T}^{t+T}\int_{\overline{B(0,T)}}
   \Gamma(\tau,y;0,x)\,\d\tau\,\d y < \infty.
  \end{align*}
  We now turn to prove assertion (b). To this end,
  let $\varphi\in C^\infty_0(\R^{1+n})$
  and let $\psi(s,y) := \varphi(-s,y)$.
  Since $\Gamma^*(w;\zeta) = \Gamma(\zeta;w)$ is a global fundamental
  solution for $\Ht^*$ (see Theorem \ref{thm.existGammastar}),
  we have
  \begin{align*}
   -\varphi(t,x) & = -\psi(-t,x)
   = \int_{\R^{1+n}}\Gamma(s,y;-t,x)\,\Ht\psi(s,y)\,\d s\,\d y \\
   &=
   \int_{\R^{1+n}}\Gamma(0,y;-t-s,x)\,\Ht\psi(s,y)\,\d s\,\d y \\
   & = \int_{\R^{1+n}}\Gamma(0,y;-t+\tau,x)\,
   \Ht\psi(-\tau,y)\,\d\tau\,\d y \\
   & \big(\text{since $(\Ht\psi)(-\tau,y) = \Ht^*\varphi(\tau,y)$}\big) \\
   & = \int_{\R^{1+n}}\Gamma(t,y;\tau,x)\,
   \Ht^*\varphi(\tau,y)\,\d\tau\,\d y \\
   & = \int_{\R^{1+n}}G(t,x;\tau,y)\,\Ht^*\varphi(\tau,y)\,\d\tau\,\d y,
   \end{align*}
   and this proves that $\Ht G(z;\cdot) = -\mathrm{Dir}_z$
   in $\mathcal{D}'(\R^{1+n})$, as desired. \medskip

   \textsc{Step II:}\,\,In this step we show that,
   for very $z=(t,x)\in\R^{1+n}$, one has
   $$\text{$G(z;\cdot)\in C(\R^{1+n}\setminus\{z\})$\,\,
   and\,\,$G(z;\zeta)\to 0$ as $\|\zeta\|\to\infty$.}$$
   On the one hand, the continuity of $G(z;\cdot)$ out of $z$
   is a direct consequence of the continuity of
   $\Gamma$ out of the diagonal;
   on the other hand, since $\Gamma(\cdot;\zeta)$
   vanishes at infinity,
   we have $$G(t,x;s,y) = \Gamma(t,y;s,x) = \Gamma(t-s,y;0,x)\longrightarrow 0,\quad \text{as $\|(s,y)\|\to \infty$.}
   $$
 Due to the uniqueness of $\Gamma$, this ends the proof.
   \end{proof}
 The next fact proves what remains to be proved of (v) in Theorem \ref{mainteo}.
  \begin{remark}\label{rem.Gammavanishtwoside}
 (1) In view of $\Gamma(t,x;s,y)=\Gamma(-s,x;-t,y)$ and the symmetry of $\Gamma$ in $x\leftrightarrow y$,
    we recognize that, for every compact set $K\subseteq\R^{1+n}$,
    $$\lim_{\|\zeta\|\to\infty}\big(\sup_{z\in K}\Gamma(\zeta;z)\big)
    =\lim_{\|\zeta\|\to\infty}\big(\sup_{z\in K}\Gamma(z;\zeta)\big) = 0.$$
    Here we used (b) of Proposition \ref{prop.regGammaHt}.\medskip

    (2)
  By Theorem \ref{thm.symmetryGammax}, it is not difficult to prove the following identity
  \begin{equation*}
   \Gamma^*(t,x;s,y) = \int_{\R^p}\Gamma_\G^*(t,x,0;s,y,\eta)\,\d\eta
  \end{equation*}
  (where $\Gamma_\G^*$ is the fundamental solution of $\Ht_\G^* = \LL_\G+\de_t$ on $\G$)
  which shows that $\Gamma^*_\G$ lifts $\Gamma^*$.\medskip

  Furthermore, by the same tricks as above,  $\Gamma^*$ satisfies
  the dual statement of Proposition \ref{prop.GammainvHtstar}, that is,
  for every $\varphi\in C^\infty_0(\R^{1+n})$, the function
  $\Lambda^*_\varphi$ defined by
  $$\Lambda^*_\varphi(\zeta) := \int_{\R^{1+n}}\Gamma^*(z,\zeta)
  \,\varphi(z)\,\d z,
  \qquad \zeta\in\R^{1+n},$$
  is well-defined and it satisfies the following properties:
 $\Lambda^*_\varphi\in C^\infty(\R^{1+n})$
   and $\Ht^*(\Lambda^*_\varphi) = -\varphi$ point-wise on $\R^{1+n}$;
   $\Lambda^*_\varphi(\zeta) \to 0$ as $\|\zeta\|\to\infty$.
\end{remark}
 Finally, the next proposition proves (x) in Theorem \ref{mainteo}.
 \begin{proposition}
 \label{prop.reproductionGamma}
 For every $x,y\in\R^n$ and every $s,t > 0$, we have the following so-called
 Re\-pro\-duction Identity:
 \begin{equation} \label{eq.reproductionGamma}
  \Gamma(0,y;t+s,x) = \int_{\R^n}\Gamma(0,w;t,x)\,\Gamma(0,y;s,w)\,\d w.
 \end{equation}
 \end{proposition}
 \begin{proof}
  We fix a point $(s,y)\in(0,\infty)\times\R^n$ and we define
  $\varphi_{s,y}(w) := \Gamma(0,y;s,w)$.
  Since $\Gamma(0,y;\cdot)$ is smooth out of $(0,y)$ and since $s > 0$,
  it is immediate to check that  $\varphi_{s,y}\in C^\infty(\RN,\R)$; moreover,
  since $\Gamma(0,y;\cdot)$ vanishes at infinity, we see that $\varphi_{s,y}$ is also
  bounded on $\R^N$. Thus, Theorem \ref{thm.solCauchyPb} implies that
  $$u(t,x) :=  \int_{\R^n}
  \Gamma(0,w;t,x)\,\varphi_{s,y}(w)\,\d w=
  \int_{\R^n}\Gamma(0,w;t,x)\,\Gamma(0,y;s,w)\,\d w $$
  is the unique bounded solution of the Cauchy problem
  $$
     \Ht u = 0 \quad \text{in $\Omega = (0,\infty)\times\R^n$,}\qquad
     u(0,x) = \Gamma(0,y;s,x) \quad \text{for $x\in\R^n$.} $$
 We now claim that the function $\overline{\Omega}\ni(t,x)\mapsto v(t,x) := \Gamma(0,y;t+s,x)$ is also a bounded
 solution of the same Cauchy problem. Indeed, since $s > 0$ is fixed, Corollary \ref{cor.GammaHtSmooth} shows
 that $v\in C^\infty(\Omega,\R)$ and that $\Ht v = 0$ on $\Omega$; moreover,
 since $\Gamma(0,y;\cdot)$ vanishes at infinity, we deduce that $v$ is bounded on $\Omega$.
 Since, obviously, $v(0,x) = \Gamma(0,y;s,x)$, we then conclude that $v \equiv u$ on
 the whole of $\Omega$, and the Reproduction
 Identity \eqref{eq.reproductionGamma} follows.
 \end{proof}	
 \bigskip

\textbf{Acknowledgements:}
 We wish to thank Marco Bramanti for useful discussions. Some of the results of this paper have been presented
 by the first-named author during the Conference ``New trends in PDEs'' (May 29-30, 2018 - University of Catania, Italy).



\end{document}